\documentclass[a4paper,12pt,twoside]{article}
\usepackage{mathrsfs,amsmath,amsfonts,amssymb,amsthm,color}
\usepackage[pdftex]{graphicx,xcolor} 
\usepackage{fancybox}
\usepackage{comment,url,esint}
\usepackage{tikz}
\usepackage{here}
\usepackage{autobreak}
\usepackage{breqn}
\usepackage{listings,jvlisting} 

\numberwithin{equation}{section}
     \newtheorem{thm}{Theorem}[section]
     
     \newtheorem{prop}[thm]{Proposition}
     
\theoremstyle{definition}
     \newtheorem{defn}[thm]{Definition}
     
\theoremstyle{remark}
     \newtheorem{rem}[thm]{Remark}

\newcommand{\ve}{\varepsilon}

\newcommand{\emax}{\operatorname{emax}}
\newcommand{\einf}{\operatorname{einf}}

\newcommand{\sizetplus}{9{,}573\ }
\newcommand{\sizetminus}{17{,}261\ }

\newcommand{\paramincsize}{{\tt INCSIZE}} 
\newcommand{\paramnshift}{{\tt NSHIFT}}
\newcommand{\vark}{{\tt K}}
\newcommand{\varl}{{\tt L}}
\newcommand{\variniemax}{{\tt iniemax}}
\newcommand{\varinieinf}{{\tt inieinf}}
\newcommand{\varemax}{{\tt emax}}
\newcommand{\vareinf}{{\tt einf}}
\newcommand{\arraysmallc}{{\tt c}}
\newcommand{\arraylargec}{{\tt C}}
\newcommand{\arraysmallinc}{{\tt inc}}
\newcommand{\arraylargeinc}{{\tt INC}}
\newcommand{\subroutinegauss}{{\tt GAUSS}}
\newcommand{\subroutineexplore}{{\tt EXPLORE}}
\newcommand{\subroutinecoarsegraining}{{\tt COARSE\_GRAINING}}
\begin{document}

\title
{Numerical experiments on the Hardy conjecture for the Gauss circle problem}
\author{Satoshi Yamaguchi, Shoichi Fujima, Shigehiko Kuratsubo, \\ Eiichi Nakai and Tsuyoshi Yoneda}
\date{}

%
%
%
%



\maketitle

\begin{abstract}
The classical unsolved Gauss circle problem concerns estimating the error 
between the number of lattice points inside a circle and the area of the circle 
as its radius tends to infinity. 
About a century ago, 
Hardy proposed a conjecture concerning this problem. 
In this paper, 
we attempt to provide numerical evidence in support of the Hardy conjecture 
through large-scale numerical computations.
\\[1ex]
\\[0.5ex]
\end{abstract}

\section{Introduction} 

The Gauss circle problem is a classical problem in number theory. 
In this paper, we investigate this problem through numerical experiments.

A point $(x,y)\in\mathbb{R}^2$ is called a lattice point if both $x$ and $y$ are integers. 
For $t>0$, let $A(t)$ denote the number of lattice points inside the circle centered at the origin with radius $r=\sqrt{t}$; 
that is, the number of integer solutions of
\begin{equation*}
 x^2+y^2\le t, 
\end{equation*}
see Figure~\ref{fig:lattice1}.
The number $A(t)$ is equal to the area of the polygon shown in Figure~\ref{fig:lattice2}, 
since the polygon is the union of the unit squares centered at the lattice points inside the circle. 
Thus, $A(t)$ is approximated by the area $\pi t$ of the circle.
\begin{figure}[htbp]
\begin{minipage}{.45\linewidth}
\begin{center}
\includegraphics[width=.80\linewidth]{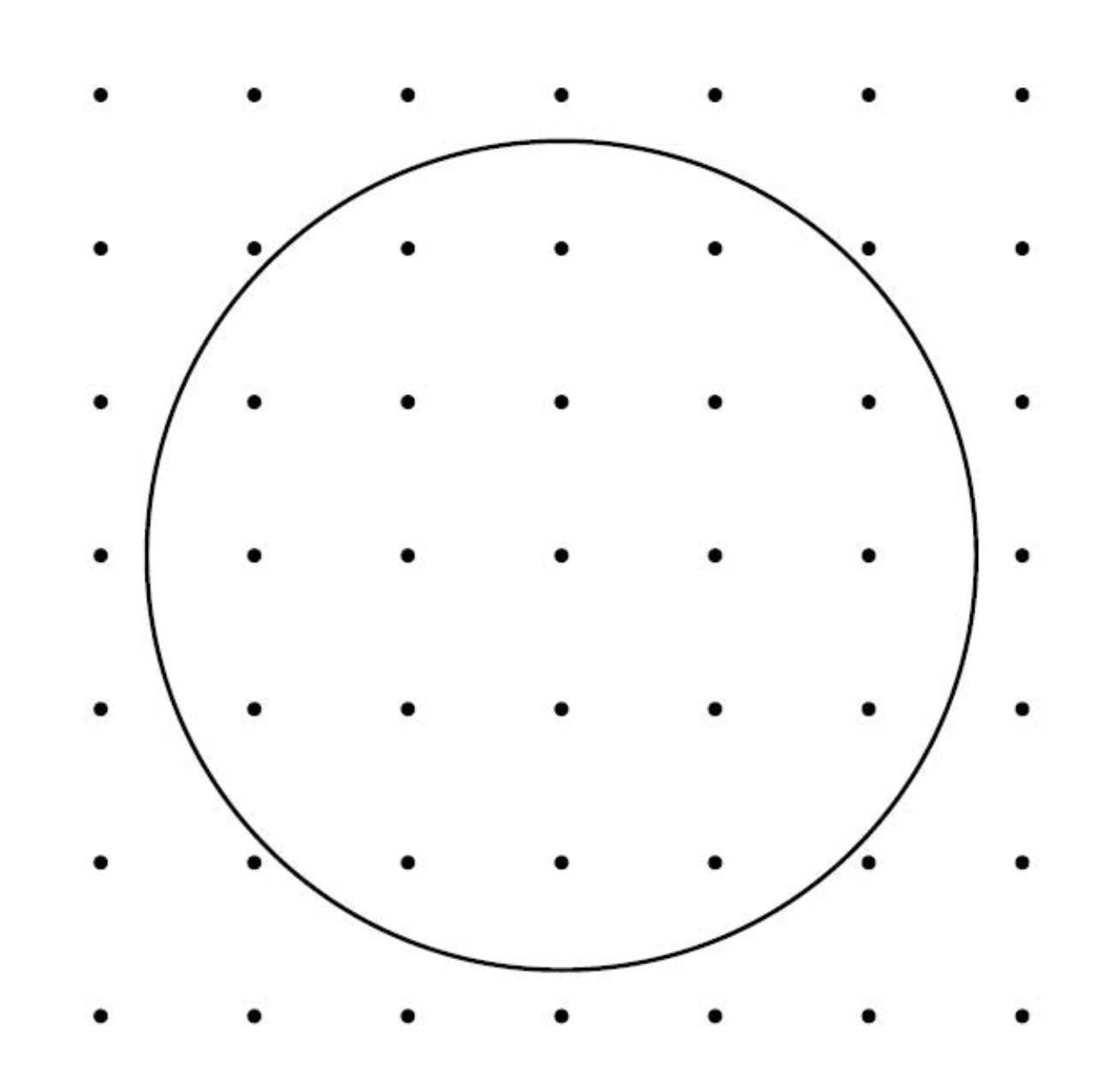} 
\caption{The Gauss circle problem}
\label{fig:lattice1}
\end{center}
\end{minipage}
\begin{minipage}{.45\linewidth}
\begin{center}
\includegraphics[width=.80\linewidth]{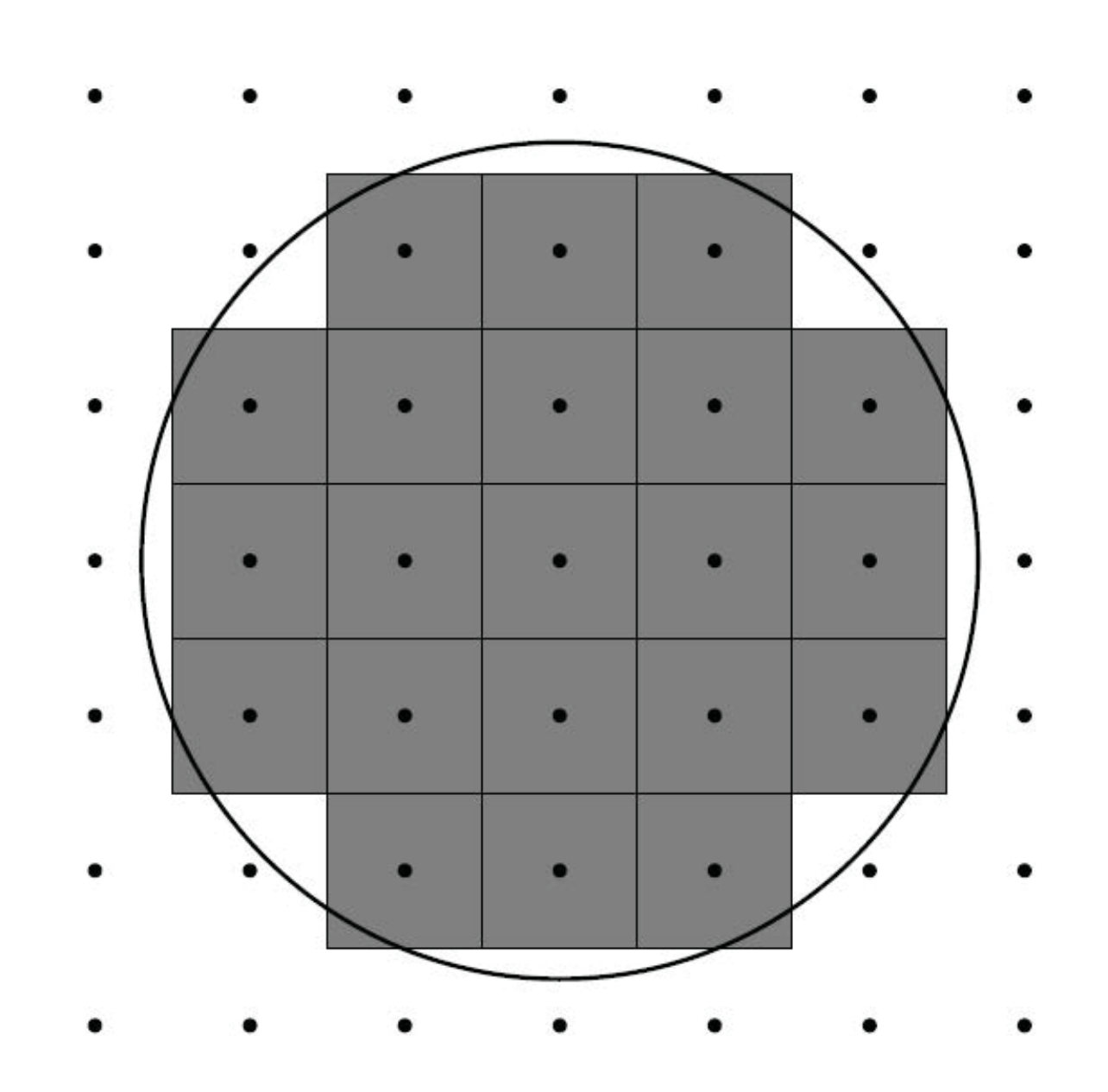} 
\caption{The union of the unit  squares.}
\label{fig:lattice2}
\end{center}
\end{minipage}
\end{figure}
Gauss showed, by a geometric argument, that
\begin{equation}\label{eq:gauss}
 A(t)=\pi t + O\bigl(t^{1/2}\bigr)\quad\text{as}\quad t\to\infty.
\end{equation}
Here, $O$ denotes Landau's symbol. 
More precisely, for a positive-valued function $g$, 
the notation $f(t)=O(g(t))$ as $t\to\infty$ means that
$\limsup_{t\to\infty}|f(t)|/g(t)<\infty$.

Let 
\begin{equation}
 E(t)=A(t)-\pi t.
\end{equation}
Determining the best possible bound on $\theta$ such that $E(t)=O(t^{\theta})$
is a famous open problem known as the Gauss circle problem.
Vorono\"i~\cite{Voronoi1904} (1904) and Sierpi\'nski~\cite{Sierpinski1906} (1906) proved that $\theta\le1/3$.
Since then, numerous studies have been carried out;
however, the best results to date remain 
those of Huxley~\cite{Huxley2003} (2003), $\theta\le\tfrac{131}{416}=0.3149\dots$, 
and Liu and Yang~\cite{Li-Yang-arXiv} (arXiv), $\theta\le0.314483\dots$.

On the other hand, Hardy~\cite{Hardy1915} (1915) proved that $\theta>1/4$; 
more precisely $E(t)\ne O(t^{1/4}(\log t)^{1/4})$.
Moreover, Hardy~\cite{Hardy1917} (1917) conjectured that 
\begin{equation}\label{Hardy conj}
 E(t)=O(t^{1/4+\ve}) \quad\text{for any}\quad \ve>0.
\end{equation}

In recent years, 
Kuratsubo et al.~\cite{Kuratsubo2008,Kuratsubo2009,Kuratsubo2010,Kuratsubo-Nakai2022,Kuratsubo-Nakai-Ootsubo2010,Ootsubo-etal2021} 
have established a close relationship between the lattice point problem and the convergence of Fourier series. 
In particular, they~\cite{Kuratsubo-Nakai2022} proved that in two dimensions, 
the convergence or divergence of Fourier series at the origin is equivalent to the Hardy conjecture \eqref{Hardy conj}.
This connection further highlights the significance of the Hardy conjecture.

Recently, Lester and Wigman~\cite{Lester-Wigman2024} 
strengthened the probabilistic evidence for the Hardy conjecture 
by analyzing the randomness of lattice points near the boundary of the circle. 
In this paper, 
we provide further numerical evidence in support of the Hardy conjecture 
for the Gauss circle problem by computing $E(t)$ explicitly using a supercomputer.

Keller and Swenson~\cite{Keller-Swenson1963} (1963) carried out computations up to $r=\sqrt{t}=259{,}750$, 
while van de Lune and Wattel~\cite{vandeLune-Wattel1990} (1990) and Tromp~\cite{Tromp1990} (1990) 
extended these computations up to $t=58{,}956{,}361{,}256$ and $t=2.29\times 10^{12}$, respectively. 
In this paper, we further extend these computations up to $t=1.22\times 10^{18}$. 
Our computational method is a refinement of Tromp's method in \cite{Tromp1990}.

Our computations indicate that $\theta<0.261$ over the range considered, 
providing further numerical evidence in support of the Hardy conjecture. 
Furthermore, we conjecture that
\begin{equation}
E(t)=O(t^{1/4}(\log t)^{2/5})
\quad\text{or}\quad
E(t)=O(t^{1/4}(\log t)^{1/3})
\quad\text{as}\quad t\to\infty.
\end{equation}

We also found that the extreme values of $E(t)$ are not uniformly distributed over the interval covered by our computations, 
but are concentrated in several narrow subintervals. 
Although these subintervals were identified through numerical experimentation, 
we do not know of any general rule for determining their locations. 
Understanding this phenomenon may require deeper number-theoretic insights.

The organization of this paper is as follows.
In Section~\ref{sec:formulation}, 
we first formulate the problem. 
More precisely, we describe the method used to investigate the extreme values of $E(t)$.
In Section~\ref{sec:program}, 
we describe the Fortran programs used for the computations. 
We employ an improved version of the ``tracker and counter array
algorithm'' and the ``skipping technique'' introduced by Tromp~\cite{Tromp1990}.
In Section~\ref{sec:conjecture}, 
based on the computational results, 
we formulate conjectures concerning the Gauss circle problem.
In Section~\ref{sec:tendency}, 
we describe the concentration of the extreme values of $E(t)$ in several narrow subintervals.
Finally, in the Appendix, 
we provide programs that use this localization property to efficiently compute all the extreme values obtained in our computations.

\section{Formulation}\label{sec:formulation}

First, by computing $A(t)$ and $E(t)$ for small $t$, we obtain the following Figure~\ref{fig:A(t)} and \ref{fig:E(t)}, respectively.
We focus on the extreme values of $E(t)$. 
More precisely, we consider
\begin{equation}\label{emax einf}
\emax(t)=\max_{s\le t} E(s)
\quad\text{and}\quad
\einf(t)=\inf_{s\le t} E(s).
\end{equation}
Then $\emax(t)$ is a step function, 
whereas $\einf(t)$ is a continuous piecewise linear function.

Let $T_+$ denote the set of all discontinuity points of $\emax(t)$, 
and let $T_-$ denote the set of all points 
at which the right derivative of $\einf(t)$ is equal to $0$ 
and the left derivative is equal to $-\pi$. 
Then
\begin{align}
 \emax(t)&=E(t), & t\in T_+, \\
 \einf(t)&=E(t-0)=\lim_{s\to t-0}E(s), & t\in T_-.
\end{align}
Let $T_\pm=T_+\cup T_-$.

Figure~\ref{fig:E(t)} shows the graphs of $\emax(t)$ and $\einf(t)$ for small $t$. 
Figure~\ref{fig:E(t)-1220p} shows the values of $\emax(t)$ and $\einf(t)$ at all $t\in T_{\pm}$ included in our computations. 
The graph shows \sizetplus points from $T_+$ and \sizetminus points from $T_-$.
From this graph, we see that
$$
\emax(t)\le -\einf(t).
$$

\begin{figure}[htbp]
\begin{minipage}{0.3\linewidth}
\begin{center}
\begin{tabular}{rr}
$t$ & $A(t)$  \\
\hline
$0$ & $1$ \\
$1$ & $5$ \\
$2$ & $9$ \\
$3$ & $9$ \\
$4$ & $13$ \\
$5$ & $21$ \\
$6$ & $21$ \\
$7$ & $21$ \\
$8$ & $25$ \\
$9$ & $29$ \\
$10$ & $37$ \\
$11$ & $37$ \\
\end{tabular}
\end{center}
\end{minipage}
\begin{minipage}{0.65\linewidth}
\begin{center}
\includegraphics[width=0.8\linewidth]{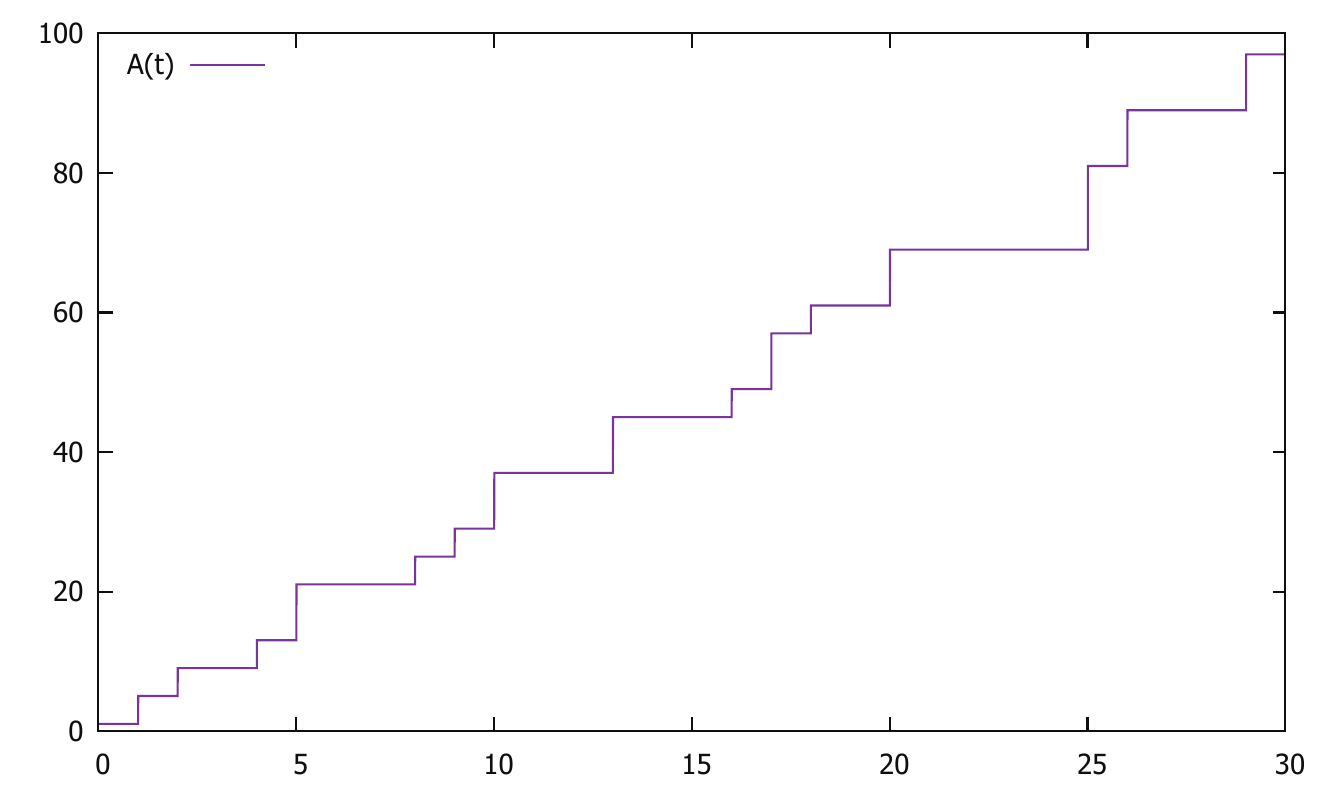} 
\end{center}
\end{minipage}
\caption{$A(t)$}
\label{fig:A(t)}
\end{figure}

\begin{figure}[htbp]
\begin{minipage}{0.3\linewidth}
\begin{center}
\begin{tabular}{lrr}
$t\in T_\pm$ & $\emax(t)$ & $\einf(t)$ \\
\hline
\phantom{0}$0$ & $1.000$ & \\
\phantom{0}$1-0$ & & $-2.142$ \\
\phantom{0}$1$ & $1.858$ & \\
\phantom{0}$2$ & $2.717$ & \\
\phantom{0}$4-0$ & & $-3.566$ \\
\phantom{0}$5$ &  $5.292$ & \\
\phantom{0}$8-0$ & & $-4.133$ \\
$10$ & $5.584$ & \\
$16-0$ & & $-5.265$ \\
$20$ & $6.168$ &  \\
$25-0$ & & $-9.540$ \\
$26$ & $7.319$ & \\
\end{tabular}
\end{center}
\end{minipage}
\begin{minipage}{0.65\linewidth}
\begin{center}
\includegraphics[width=.80\linewidth]{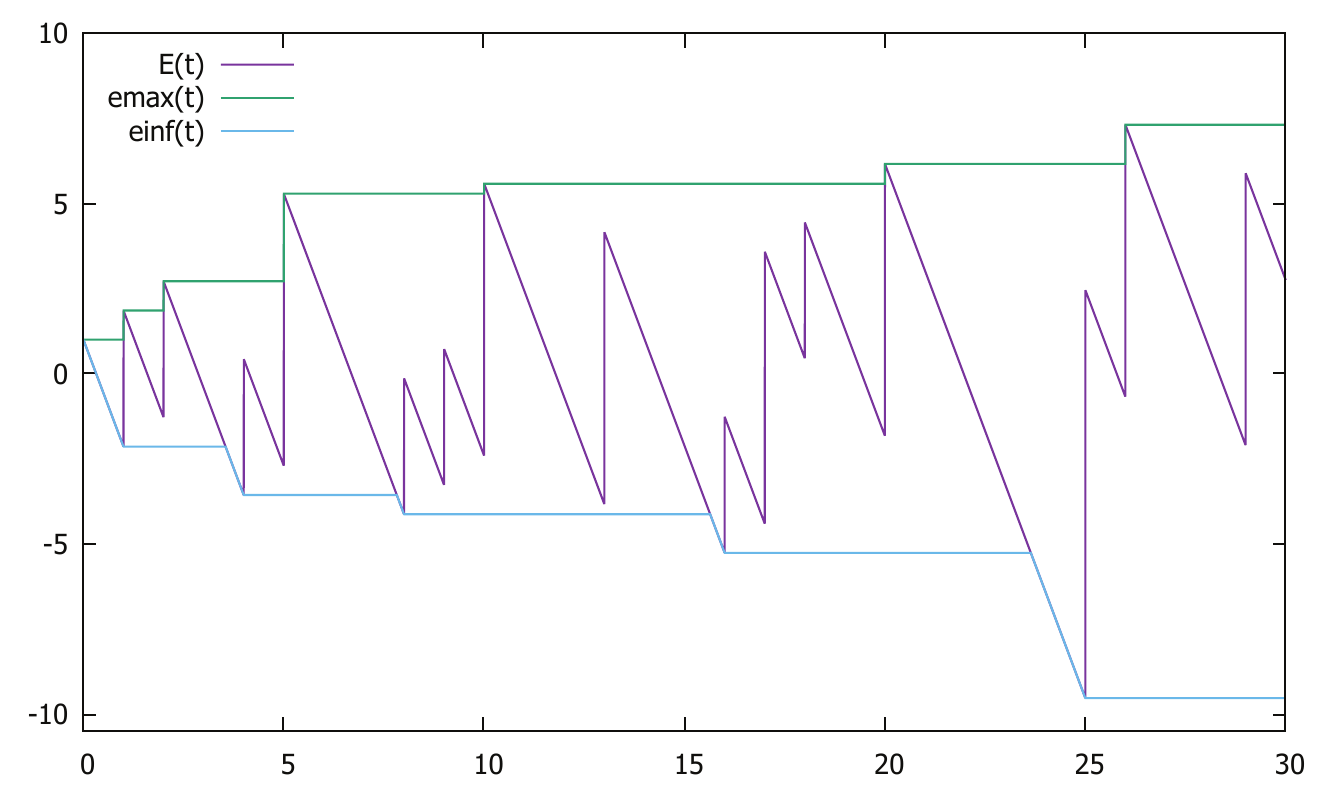} 
\end{center}
\end{minipage}
\caption{$E(t)=A(t)-\pi t$, $\emax(t)$, $\einf(t)$}
\label{fig:E(t)}
\end{figure}

\begin{figure}[htbp]
\centering
\includegraphics[width=0.8\linewidth]{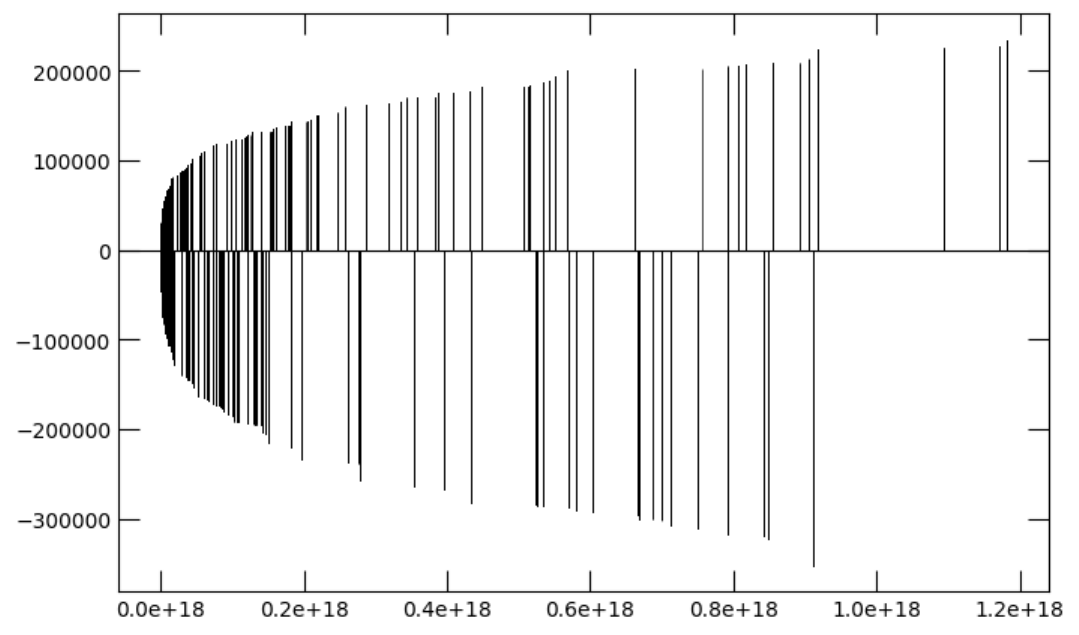}
\caption{$\emax(t)$, $\einf(t)$ at $t\in T_\pm \cap [0,1.22\times10^{18}]$}
\label{fig:E(t)-1220p}
\end{figure}

\newpage

\section{Program}\label{sec:program} 
\subsection{Count of $A(t)$}
In order to count $A(t)$, Tromp \cite{Tromp1990, vandeLune-Wattel1990} proposed ``tracker and counter array algorithm'',
which won over previous ones. 
The algorithm assumes $r$, radius of a circle, as an integer, and computes $2r+1$ values $A(r^2)$, $A(r^2+1),\cdots,A((r+1)^2-1)$
in a single ``pass''. 
A tracker starts at the lattice point $(r,0)$ and visits the all lattice points $(x,y)$ with $0\le y\le x$ 
and $r^2\le x^2+y^2<(r+1)^2$.
An array of $2r+1$ counters $\arraysmallc[0],\cdots,\arraysmallc[2r]$ 
counts the number of lattice points between two circles
for each values of $x^2+y^2-r^2$.
Also, $A(r^2-1)$ is calculated from the shape of the track. Then,
\[
	A(t)=A(r^2-1)+\sum_{i=0}^{t-r^2}\arraysmallc[i],	\quad\quad(t=r^2,\cdots,(r+1)^2-1).
\]

Tromp \cite{Tromp1990} also proposed ``skipping
technique'', where $k$ consecutive counters
$\arraysmallc[ik], \cdots,$ $\arraysmallc[(i+1)k-1]$ are merged into  one new counter
$\arraylargec[i]$, and he added 8 to $\arraylargec[d\,\mbox{div}\,k]$ instead of adding
8 to $\arraysmallc[d]$. 

Our method to count $A(t)$ is a modified version of Tromp's method. We improve following points.
\begin{enumerate}
\item For arbitrary $t_0(\ge 0)$ and $t_1(\ge t_0)$, we compute  $A(t_0), A(t_0+1), \cdots, A(t_1)$ 
in a single ``pass''.
We use trackers of two types. Main tracker starts at the points $(\varl,0)$, where $\varl=\lfloor\sqrt{t_0}\rfloor$ and $\lfloor x \rfloor$ is the greatest integer less than or equal to $x$.
The tracker goes up if current point $(x,y)$ is $x^2+y^2\le t_0$ and goes left otherwise until it arrives on the line
$x=\vark$, where $\vark=\lfloor\sqrt{t_0/2}\rfloor$ (see Figure~\ref{fig:maintracker}). 
On each horizontal grid line, a sub tracker starts from the position of main tracker and goes right until it leaves area $x^2+y^2\le t_1$ (see Figure~\ref{fig:subtrackers}).
Trackers preserve integer function value $x^2+y^2-t_0-1$ while moving,
 so that 
determining whether trackers are inside or outside two circles,
$C(t_0)$ and $C(t_1)$, can be exact
and computationally inexpensive
(concerning the function value of the main tracker's position,
see Figure~\ref{fig:maintracker}).

From shape of the path of the main tracker we know $A(t_0)$, while $A(t_0+1), \cdots, A(t_1)$
are calculated using an array of $t_1-t_0$ counters,
$\arraysmallinc[0], \cdots, \arraysmallinc[t_1-t_0-1]$,
\begin{equation}
	A(t)=A(t_0)+4\sum_{i=0}^{t-t_0-1}\arraysmallinc[i],	\qquad	(t=t_0+1,\cdots,t_1).
	\label{eqn:summingcounters}
\end{equation}
\item In the case we employ skipping technique, 
we merge $2^k$ consecutive counters $\arraysmallinc[i2^k]$, $\cdots$, 
$\arraysmallinc[(i+1)2^k-1]$ into new
counter $\arraylargeinc[i]$.
Since we can take $t_0$ and $t_1$ as arbitrary positive integers,
we take $t_1-t_0$ as a multiple of $2^k$
and merge counters naturally.
The size of array $\arraylargeinc[\cdot]$ is $\paramincsize = (t_1-t_0)/2^k$.
We specify
an address of merged counter by $\arraylargeinc[{\tt ISHFT}(i,-k)]$,
where {\tt ISHFT} is the bit shift function in FORTRAN.
\end{enumerate}
\begin{figure}[htbp]
\begin{minipage}{.48\linewidth}
\begin{center}
\includegraphics[width=.65\linewidth]{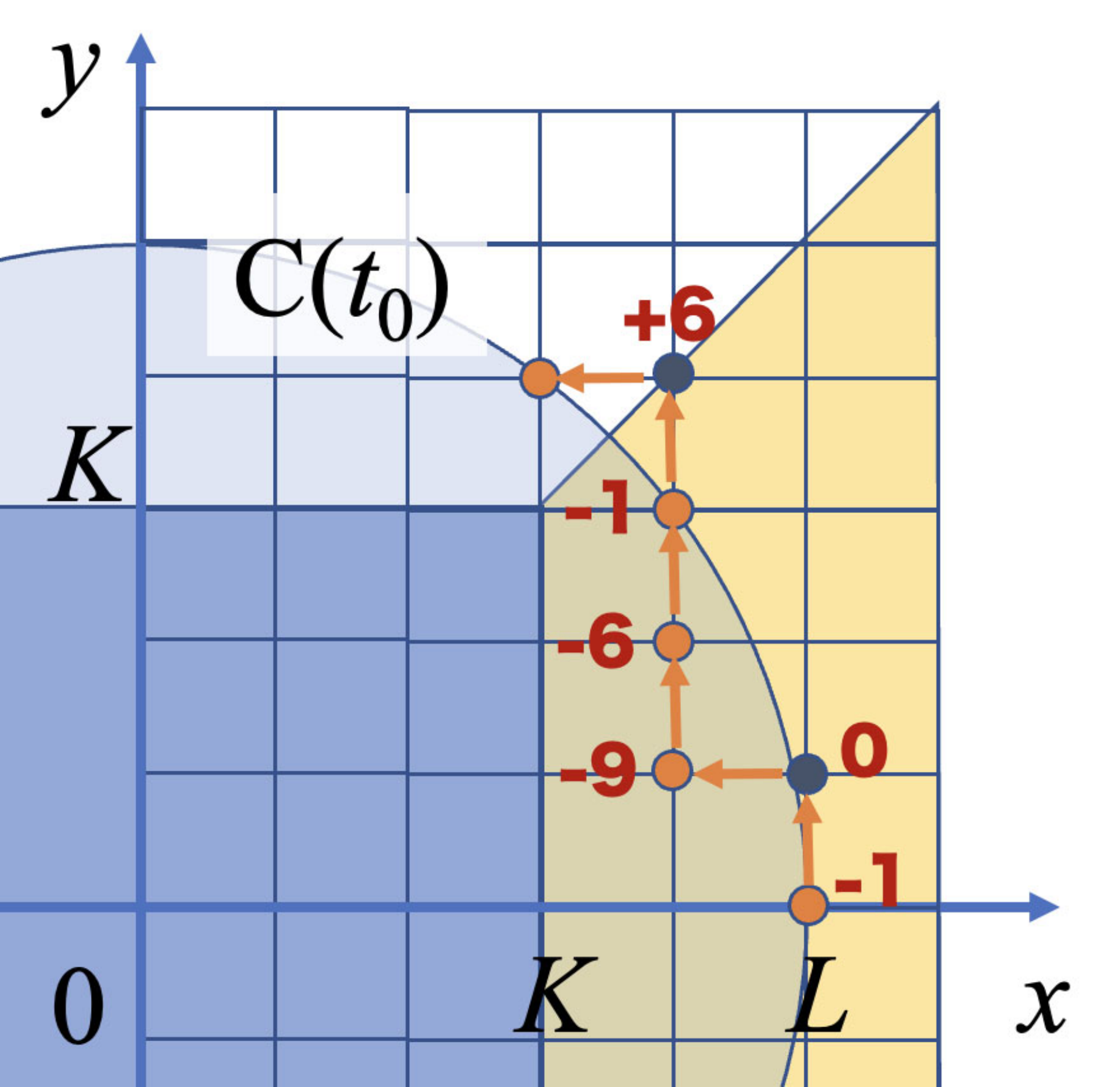}
\caption{Path of the main tracker, $t_0=25$}
\label{fig:maintracker}
\end{center}
\end{minipage}
\quad
\begin{minipage}{.48\linewidth}
\begin{center}
\includegraphics[width=.80\linewidth]{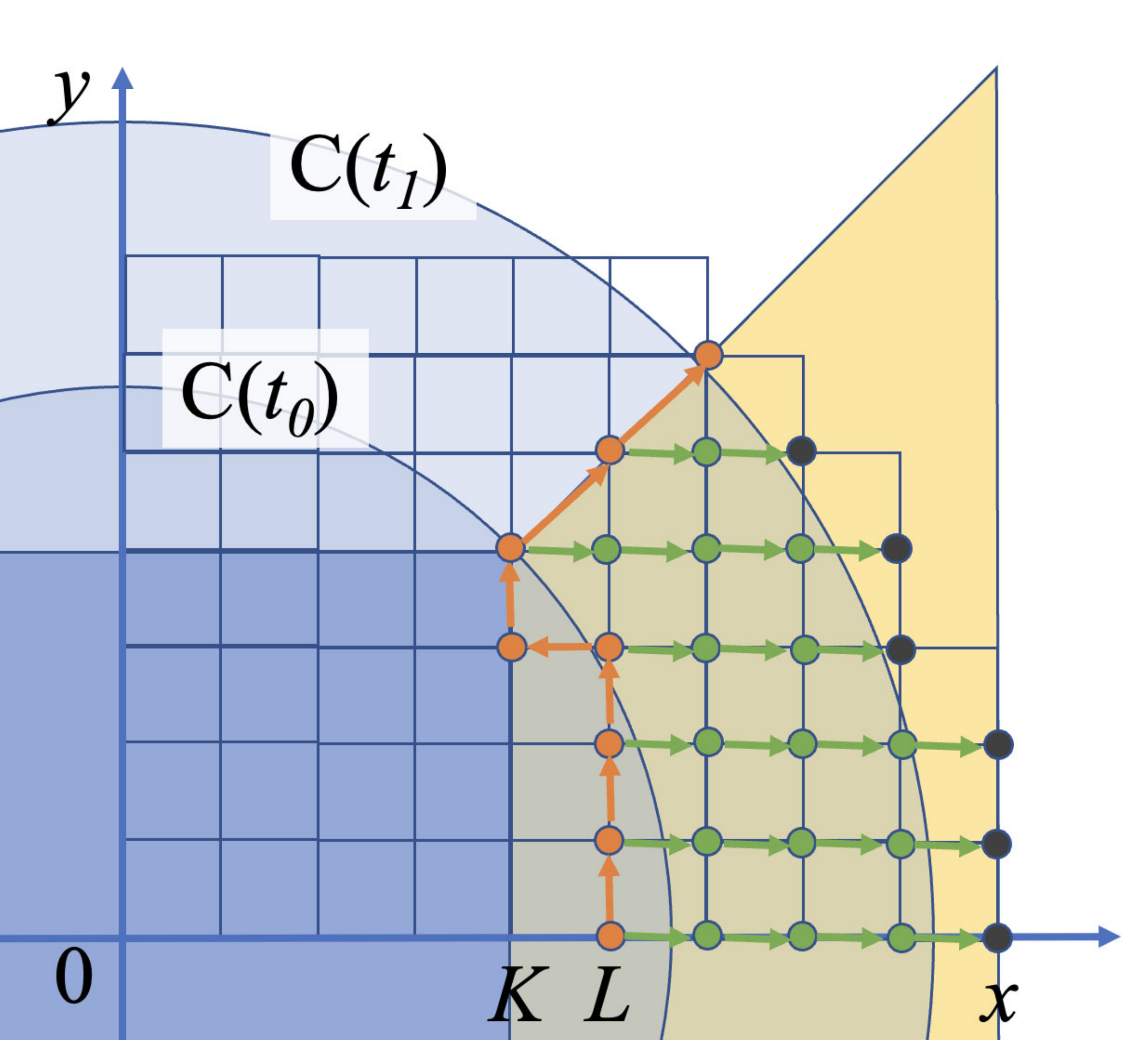}
\caption{Main and subtrackers, $t_0=32$, $t_1=70$}
\label{fig:subtrackers}
\end{center}
\end{minipage}
\end{figure}
We show our FORTRAN90 subroutine ``\subroutinegauss'' to count $A(t_0)$ and to obtain array $\arraysmallinc[\cdot]$
by above method in Listing 1.
Parameters, \paramincsize\ and \paramnshift\ are provided by Module {\tt MY_PARAMETERS}.

In the normal usage (that is, the skipping technique is not used), 
parameter \paramincsize\ is set greater than or equal to $t_1-t_0$,
and, $\paramnshift = 0$. 

In the case the skipping technique is used, if we merge $2^k$
$\arraysmallinc[\cdot]$ counters to a single entry in 
$\arraylargeinc[\cdot]$,
then those parameters are set as $\paramincsize\ge(t_1-t_0)2^{-k}$ and
$\paramnshift=k$.
The skipping technique is further explained in Section 3.3.
\begin{lstlisting}[caption=Subroutine \subroutinegauss\  to count $A(t_0)$ ({\tt gauss.f90})]
SUBROUTINE GAUSS( T0, T1, A, INC )	
  USE MY_PARAMETERS	
  IMPLICIT NONE	
  INTEGER(8), INTENT(IN):: T0, T1
  INTEGER(8), INTENT(OUT):: A
  INTEGER, INTENT(OUT):: INC(0:INCSIZE-1)	
  INTEGER(8):: S, STEP_T, K, L, I, IK, J, IJ,   I1, IJ1
  REAL(16), PARAMETER:: ROOTHALF = SQRT(0.5Q0) 	
  REAL(16) ROOT                       
  ROOT = SQRT( REAL(T0, 16) + 0.5Q0 )  
  INC = 0	
  L = INT( ROOT )	
  K = INT( ROOT * ROOTHALF )
  STEP_T = T1 - T0

  !! At (x,y) we have I=2x-1, J=2y-1 and IJ = x^2+y^2-T0-1
  IK = 2 * K + 1		!!   Value of I at x=K+1
  
  !! Now, we put main tracker at (x,y)=(L,0).
  I = 2 * L - 1 ; J = - 1 ; IJ = L * L - T0 - 1     
  
  !! A sub-tracker starts from (x1,y1)=(L+1,0) and goes right.
  I1 = I + 2 ; IJ1 = IJ + I1  
  DO WHILE (IJ1 < STEP_T)   !! <=> While (x1^2+y1^2 <= T1)
     INC(ISHFT(IJ1,-NSHIFT)) = 1
     I1 = I1 + 2 ; IJ1 = IJ1 + I1
  END DO
  S = 0  !! Clear counter of lattice pts x>K and x^2+y^2<=T0
  DO WHILE (I >= IK)	!! <=>(Equivalent to) While (x>K)
     DO	
        J = J + 2 ; IJ = IJ + J	         !! Main tracker goes up
        !! A sub-tracker starts from (x1,y1)=(x+1,y) and goes right        
        I1 = I + 2 ; IJ1 = IJ + I1                     
        DO WHILE (IJ1 < STEP_T) !! <=> While (x1^2+y1^2 <= T1)
           INC(ISHFT(IJ1,-NSHIFT)) = INC(ISHFT(IJ1,-NSHIFT)) + 2	
           I1 = I1 + 2 ; IJ1 = IJ1 + I1
        END DO
        
        IF (IJ >= 0) EXIT	!! <=> If (x^2+y^2 > t0)
     END DO	
     IF (IJ < STEP_T) THEN	!! <=> If (x^2+y^2 <= T1)
        IF (J < I) THEN	
           INC(ISHFT(IJ,-NSHIFT)) = INC(ISHFT(IJ,-NSHIFT)) + 2	
        ELSE	
           INC(ISHFT(IJ,-NSHIFT)) = INC(ISHFT(IJ,-NSHIFT)) + 1	
        END IF	
     END IF	
     S = S + J !! S counts lattice points in x^2+y^2 <= t and x > K
     IJ = IJ - I ; I = I - 2	 !! Main tracker goes left
  END DO	
  
  !! Now, main tracker arrives at x = K, and y = K-2, K-1, K or K+1. 
  DO WHILE (J < I)	!! <=> While (y<K)
     J = J + 2 ; IJ = IJ + J	!! Main tracker goes up

     !! A sub tracker starts from (x1,y1)=(K+1,y) and goes right
     I1 = I + 2 ; IJ1 = IJ + I1	
     DO WHILE (IJ1 < STEP_T)	!! <=> While (x1^2+y1^2 <= T1)
        INC(ISHFT(IJ1,-NSHIFT)) = INC(ISHFT(IJ1,-NSHIFT)) + 2	
        I1 = I1 + 2 ; IJ1 = IJ1 + I1	
     END DO	
  END DO	
  
  !! Now, main tracker is at (x,y)=(K,K) or (K,K+1)
  I = I + 2 ; IJ = IJ + I	!! Main tracker goes right (x becomes K+1)
  IF (J < I) THEN !! <=> If (y == K)	
      J = J + 2 ; IJ = IJ + J	!! Main tracker goes up (y becomes K+1)
      IF (IJ < STEP_T) THEN	!! <=> If ((K+1)^2+(K+1)^2 <= T1)
          INC(ISHFT(IJ,-NSHIFT)) = INC(ISHFT(IJ,-NSHIFT)) + 1

	 !! Sub tracker starts from (x1,y1)=(K+2,K+1) and goes right          	
          I1 = I + 2 ; IJ1 = IJ + I1	
          DO WHILE (IJ1 < STEP_T)	!! <=> While (x1^2+y1^2 <= T1)
              INC(ISHFT(IJ1,-NSHIFT)) = INC(ISHFT(IJ1,-NSHIFT)) + 2
              I1 = I1 + 2 ; IJ1 = IJ1 + I1	
          END DO	
      END IF	
  END IF	
  
  !! Now, main tracker is at (x,y)=(K+1,K+1)
  DO
     I = I + 2 ; J = J + 2 ; IJ = IJ + I + J !! Goes upper right
     IF (IJ >= STEP_T) EXIT  !! <=> If (x^2+y^2 > T1)
     INC(ISHFT(IJ,-NSHIFT)) = INC(ISHFT(IJ,-NSHIFT)) + 1

     !! Sub tracker starts from (x1,y1)=(x+1,y) and goes right
     I1 = I + 2 ; IJ1 = IJ + I1
     DO WHILE (IJ1 < STEP_T) !! <=> While (x1^2+y1^2 <= T1)
        INC(ISHFT(IJ1,-NSHIFT)) = INC(ISHFT(IJ1,-NSHIFT)) + 2
        I1 = I1 + 2 ; IJ1 = IJ1 + I1
     END DO
  END DO
  A = INT(IK, KIND=8) * INT(IK, KIND=8) + 4 * S
END SUBROUTINE GAUSS
\end{lstlisting}
\subsection{Detection of elements of $T_\pm$}
In order to detect elements of $T_\pm$ in $t_{\rm FIRST}< t\le t_{\rm LAST}$,
the standard approach is to prepare values of $\emax(t_{\rm FIRST})$ and $\einf(t_{\rm FIRST})$
and to check the changes of functions of $\emax(t)$ and $\einf(t)$
while increasing $t$. 

Listing 2 shows FORTRAN90 subroutine ``\subroutineexplore'' to detect elements of $T_\pm$
in the interval $(t_{\rm FIRST}, t_{\rm LAST}]$.
The interval are divided into subintervals whose size is \paramincsize,
except for the last subinterval whose size may be less than \paramincsize.
For each subinterval $(t_0,t_1]$, \subroutinegauss\ subroutine in the normal usage works, and 
$A(t_0)$ and counters $\arraysmallinc[i]$ are obtained, where $1\le i\le t_1-t_0$.
Since points $T_\pm$ only appear at $t$ where $A(t)-A(t-0)>0$,
that can be known by $\arraysmallinc[t-t_0]>0$, we investigate such $t$'s.

In addition to detection of $T_\pm$, we check exact match of 
$A(t_0)$ with {\tt LAST_A}, where {\tt LAST_A}
is $A(T_0)$ that was obtained by summing up 
(\ref{eqn:summingcounters}) in the previous subinterval.
This check is useful to secure reliability of the counting $A(t)$.
As a matter of fact, we had sensed several unexpected errors
by this check, and had fixed bugs in the programs.
\begin{lstlisting}[caption=Subroutine \subroutineexplore\ to detect elements of $T_\pm$ ({\tt explore.f90})]
SUBROUTINE EXPLORE(FIRST_T, LAST_T, EMAX, EINF, CHK_EMAX, CHK_EINF, NEXS)
  USE MY_PARAMETERS
  IMPLICIT NONE
  INTEGER(8), INTENT(IN):: FIRST_T, LAST_T
  REAL(8), INTENT(INOUT):: EMAX, EINF
  LOGICAL, INTENT(IN):: CHK_EMAX, CHK_EINF
  INTEGER, INTENT(OUT):: NEXS ! Total number of T_\pm detected in the interval

  REAL(16), PARAMETER:: PI = ATAN(1.0Q0)*4.0Q0
  INTEGER(8) T0, T1, T, A, LAST_A
  REAL(8) E
  REAL(16) REAL_T, ROOT_T
  INTEGER, DIMENSION(INCSIZE):: INC !! Index=1,...,INCSIZE in this subroutine
  INTEGER I
  
  NEXS = 0
  T0 = FIRST_T
  DO
     T1 = T0 + INCSIZE
     IF (T1 > LAST_T) T1 = LAST_T

     CALL GAUSS( T0, T1, A, INC)
     
     ! Check that A(T0) == previous A(T1)
     IF (T0 > FIRST_T .AND. A /= LAST_A) THEN
        PRINT *,"ERROR: A(T) DO NOT MATCH AT T=",T0,", A=",A,", LAST_A=",LAST_A
        STOP
     END IF
     IF (T0 == LAST_T) EXIT
     
     DO I = 1, T1 - T0
        IF (INC(I) > 0) THEN
           T = T0 + I ; REAL_T = REAL(T, 16)
           IF (CHK_EINF) THEN
              E = REAL(A, 16) - PI * REAL_T
              IF (E < EINF) THEN
                 EINF = E ; NEXS = NEXS + 1
                 WRITE (*,"('T = ',I19,'(-0)  A = ',I19,'   E = ',F12.3)") T, A, EINF
              END IF
           END IF
           A = A + INC(I) * 4
           ! PRINT *,"T=",T,", A=",A
           IF (CHK_EMAX) THEN
              E = REAL(A, 16) - PI * REAL_T
              IF (E > EMAX) THEN
                 EMAX = E ; NEXS = NEXS + 1
                 !WRITE (*,"('T = ',I19,'      A = ',I19,'   E = ',F12.3)") T, A, EMAX
                 ROOT_T = SQRT(REAL_T)-1E-7
                 WRITE (*,"('T = ',I19,'      E = ',F12.3,' POS = ',F7.4)") T, EMAX, ROOT_T-FLOOR(ROOT_T)
              END IF
           END IF
        END IF
     END DO
     LAST_A = A
     T0 = T1
  END DO
END SUBROUTINE EXPLORE
\end{lstlisting}

\subsection{Coarse-graining of elements of $T_\pm$}
In Subroutine \subroutineexplore, calculations of $E(t)$ at all $t$'s
where $A(t)-A(t-0)>0$ regardless that  $|E(t)|$ is large or not
are computationally expensive.
The coarse-graining is effective to reduce computation costs.

Let $\Delta t$ be a positive integer and consider that we only compute $A(t)$ every $\Delta t$
employing the skipping technique in Subroutine \subroutinegauss.
Now, let consider we know $A(t_1)$ and $A(t_2)$, ($t_2=t_1+\Delta t$).
It is clear that
\[
	E(t_1)-(t-t_1)\pi \le E(t) \le E(t_2)+(t_2-t)\pi	,	\qquad t_1\le t\le t_2.
 \]
(see Figure~\ref{fig:E(t)}).
From this inequality we have
\begin{align*}
	&T_+\cap(t_1,t_2]=\emptyset
	 \qquad \mbox{if }E(t_2)\ \le\ \emax(t_1)-(\Delta t-1)\pi,
	 \\
	&T_-\cap(t_1,t_2]=\emptyset
	 \qquad \mbox{if }E(t_1)\ \ge\ \einf(t_1)+(\Delta t)\pi,
\end{align*}
so that we can extract ``$T_\pm$-able intervals'', where elements of $T_\pm$ may 
exist in the intervals.

Listing 4 shows FORTRAN Subroutine ``\subroutinecoarsegraining'' by the idea.
The subroutine outputs $T_\pm$-able intervals. Those intervals 
should be explored afterwards by Subroutine \subroutineexplore.
Listing 3 shows a tiny operation example of the coarse-graining,
where $t_{\rm FIRST}=0$, $t_{\rm LAST}=30$, 
and $\Delta t=2^1$ ($\paramnshift=1$) are set.
We see that $T_+\cap\left\{ (10,16]\cup(20,24]\cup(26,30] \right\}
=T_-\cap\left\{ (20,22]\cup(26,30] \right\}=\emptyset$ from the result.
\begin{lstlisting}[caption=Operating example of Subroutine \subroutinecoarsegraining \ ({\tt testcoarse.f90})]
! to compile: $gfortran testcoarse.f90 coarse.f90 gauss.f90
!
MODULE MY_PARAMETERS
  IMPLICIT NONE
  INTEGER(4), PARAMETER:: INCSIZE = 5
  INTEGER(4), PARAMETER:: NSHIFT = 1
END MODULE MY_PARAMETERS

PROGRAM MAIN
  REAL(8):: EMAX=0.0_8, EINF=0.0_8
  CALL COARSE_GRAINING(0_8, 30_8, EMAX, EINF)
END PROGRAM MAIN
!
! Execution result (Delta t=2):
! EMAX                   0                  10
! EINF                   0                  20
! EMAX                  16                  20
! EINF                  22                  26
! EMAX                  24                  26
\end{lstlisting}
\begin{lstlisting}[caption=Subroutine \subroutinecoarsegraining \ ({\tt coarse.f90})]
SUBROUTINE COARSE_GRAINING(FIRST_T,LAST_T,EMAX,EINF)
  !! Operating conditions: MOD(LAST_T-FIRST_T, INCSIZE*(2^NSHIFT) ) = 0
  USE MY_PARAMETERS
  IMPLICIT NONE
  INTEGER(8), INTENT(IN):: FIRST_T, LAST_T
  REAL(8), INTENT(INOUT)::  EMAX, EINF
  REAL(16), PARAMETER:: PI = ATAN(1.0Q0)*4.0Q0
  INTEGER(8), PARAMETER:: DT=ISHFT(1,NSHIFT)
  INTEGER(8), PARAMETER:: STEP_T=DT*INCSIZE
  REAL(8) E,THRESHMAX,THRESHINF
  INTEGER(8) T0, T1, T, T1MAX, T1INF, T2MAX, T2INF, A, PREV_A
  REAL(16) REAL_T
  INTEGER, DIMENSION(0: INCSIZE-1):: INC !! Index=0...,INCSIZE-1
  INTEGER I

  IF (MOD(LAST_T-FIRST_T,STEP_T) /= 0) THEN
     PRINT *,"OPERATING CONDITION ERROR"
     STOP
  END IF
  THRESHMAX = EMAX -(DT-1) * PI ; THRESHINF = EINF + DT * PI
  T1INF = -1 ; T1MAX = -1
  !
  DO T0 = FIRST_T, LAST_T, STEP_T
     T1 = T0 + STEP_T
     IF (T1 > LAST_T) T1 = LAST_T
     
     CALL GAUSS(T0, T1, A, INC)

     IF (T0 == FIRST_T) THEN
        REAL_T = REAL(T0,16) ; E = REAL(A,16) - PI * REAL_T
     ELSE IF (A /= PREV_A) THEN
        PRINT *, "ERROR: A(T) DO NOT MATCH AT T=",T0
        STOP
     END IF
     IF (T0 == LAST_T) EXIT

     T = T0
     DO I = 0, INCSIZE-1
        IF (E < THRESHINF) THEN
           IF (T1INF < 0) THEN !! If new (T_)-able interval
              T1INF = T !! Attention: +1 is removed
           END IF
           T2INF = T + DT 
           IF (E < EINF) THEN !! If einf() has new value at T
              EINF = E ; THRESHINF = EINF + DT * PI
           END IF
        ELSE
           !! Output a completed (T_)-able interval
           IF (T1INF >= 0) THEN 
              PRINT "(A4,2I20)",'EINF',T1INF,T2INF
              T1INF = -1
           END IF
        END IF
        T = T + DT ; A = A + INC(I) * 4
        !PRINT *,"T=",T,", A=",A
        REAL_T = REAL(T,16) ; E = REAL(A,16) - PI * REAL_T ! new A(T)
        IF (E > THRESHMAX) THEN
           IF (T1MAX < 0) THEN !! If new (T+)-able interval
              T1MAX = T-DT !! Attention: +1 is removed
           END IF
           T2MAX = T
           IF (E > EMAX) THEN !! If emax() has new value at T
              EMAX = E ; THRESHMAX = EMAX - (DT-1) * PI
           END IF
        ELSE
           !! Output a complete (T+)-able interval
           IF (T1MAX >= 0) THEN
              PRINT "(A4,2I20)",'EMAX',T1MAX,T2MAX
              T1MAX = -1
           END IF
        END IF
     END DO
     PREV_A = A
  END DO
  IF (T1MAX >= 0) THEN
     PRINT "(A4,2I20)",'EMAX',T1MAX,T2MAX
  END IF
  IF (T1INF >= 0) THEN
     PRINT "(A4,2I20)",'EINF',T1INF,T2INF
  END IF
END SUBROUTINE COARSE_GRAINING
\end{lstlisting}
\subsection{Implementation on supercomputers}
Most computations have been done on supercomputers. 
We use flat MPI parallel computer environment. 
If we compute $E(t)$, for $t\in(T_1,T_2]$ on computer having $n$ cores, 
the interval is divided into $n$ subintervals whose length are approximately equal:
\[
	(T_1,T_2]
	=\left(t\,_{\rm FIRST}^0,t\,_{\rm LAST}^0\right]\cup\cdots
	\cup\left(t\,_{\rm FIRST}^{n-1},t\,_{\rm LAST}^{n-1}\right]
\]
and $i$th core  computes about $\left(t\,_{\rm FIRST}^i,t\,_{\rm LAST}^i\right]$
using Subroutine \subroutineexplore\ or
Subroutine \subroutinecoarsegraining\ independently and in parallel.

In the situation, 
we have not yet known $\emax(t\,_{\rm FIRST}^i)$ and $\einf(t\,_{\rm FIRST}^i)$,
$i=0,\cdots,n-1$. 
We therefore use ``loose ${\tt emax}$/${\tt einf}$ argument setting and post-processing approach'',
that is, we provide $\emax(T_1)$ or smaller quantity as {\tt emax} argument of \subroutineexplore\ subroutine (or \subroutinecoarsegraining\ subroutine) and 
$\einf(T_1)$ or larger quantity as {\tt einf} argument for the subroutine instead
of them. As a result, many ``fake'' $T_\pm$ are reported from cores. 
We extract ``true'' $T_\pm$ in post-processing.

\subsection{Job data table}
We have proceeded with the calculation for progressively larger value of $t$ seamlessly
for over the past few years. This progress reached $t=1.22\time10^{18}$ in March 2026.

Since job execution time on the supercomputer is limited to 48 hours, for example,
we had set appropriate $t$-intervals and had submitted each interval into a job. 
We had days when we computed a sequence of intervals in parallel.
In this case also, we use the loose {\tt emax}/{\tt einf} argument setting and post-processing approach.

Job table including details of the jobs, input data and output results are presented
in the electric appendix.
An example of key items in the job table is shown in Table 1.
\begin{table}[htbp]
\begin{center}
\caption{An example of key items in the job table}
\begin{tabular}{|c|c|c|r|c|c|c|c|}
\hline
$T_1$&    &     &\variniemax&detected&&&$s\in S_{\max}$\\
$T_2$&job&files&\varinieinf&$T_\pm$&\varemax	&\vareinf 
				&or $ S_{\inf}$\\
\hline
4.5P 	&I-68&[L]	 & 48343  &
			4 652 163G& &-83287& 68206769 \\
			\cline{5-8}
4.89P&	&[S] &               &
                         4 739 794G&54538&   & 68846165\\
			\cline{5-8}
          &&&&      4 861 592G&  &-85950& 69725124\\               
\hline
4.89P &I-69& [L] & 48000 & &&&	\\
4.9012P &   &[S] &            &&&&\\
\hline
4.9012P &I-70& [L] & 54538  &
                         4 902 366G & 58574 &  & 70016900\\
                         \cline{5-8}
5.6P      &&        [S]  &&
   	               5 573 059G &   & -86707& 74652925\\ 
\hline
\multicolumn{8}{|c|}{$\cdots$ $\cdots$ $\cdots$}\\
\hline
204P & II-563&[CO]&140000&
                        204 171 038G & 144235 && 451852895\\
206P &          &[L]    &-230000 &&&&\\
\hline
206P & II-929&[CO]&140000 &&&&\\
208P &          &[L]   &-230000 &&&&\\
\hline
208P & II-288&[CO]&140000&
                        208 821 875G & 145518 && 456970323\\
210P &         &[L]    &-230000 &&&&\\
\hline
\end{tabular}
\end{center}
\end{table}

In each row, the job computed interval $t\in(T_1,T_2]$,
displayed in petas ($10^{15}$).
Quantity \variniemax\ is value
$\emax(T_1)$ or a lower value set for the job.
Quantity \varinieinf\ is the counterpart for $\einf(T_1)$,
however if its column is empty, we substituted $-(\variniemax)$ for \varinieinf.
File ``[L]'' is a log file of the job, 
and candidates of $T_\pm$ detected in the job are recorded
in file ``[S]'', however some filtering are required
if correct $\emax(T_1)$ and $\einf(T_1)$ were not provided for the job.  
If jobs employed the coarse-graining method, ``[CO]'' file
appears. Detected $T_\pm$-able intervals are recorded in the
file. And post-processing computation had been done. 

Detected $T_\pm$, displayed in gigas ($10^9$),
 are shown in the column.
If many $T_\pm$ were detected in a cluster,  
we displayed a single item. 
In \varemax\ and \vareinf\ columns, values $\emax(T_2)$ and $\einf(T_2)$ are shown if these values have updated to new
values in the job.
In $s$ column, we showed elements of $S_{\max}$
or $S_{\inf}$.
Clustering of $T_\pm$ and definitions of $S_{\max}$
and $S_{\inf}$ will be discussed in Section 5.

The computation on the largest $t$ interval was
done on ``Miyabi-C'' in JCAHPC (Joint Center for Advanced High Performance Computing).  
It has 190 nodes. One node has two CPUs of Intel Xeon CPU Max 9480
and 128GB memory. Each CPU has 56 cores.
We therefore computed 112 subintervals in parallel.
Each core can have 1GB memory ($\approx$ 128GB/112 core),
so we set $\paramincsize=2^{26}$ (the size of array \arraylargeinc\  is 128MiB) and the coarse-graining 
 method with $\Delta t=2^{12}$ is used.

\section{Conjecture within the range of our computations}\label{sec:conjecture}

Our computations yielded \sizetplus elements of $T_+$ and \sizetminus elements of $T_-$.
Based on these data, we formulate the following conjectures.

Figures~\ref{fig:emax-exp} and  \ref{fig:einf-exp} display the exponents $a_+$ and $a_-$
in the estimate $\emax(t)\sim t^{a_+}$ and $\einf(t)\sim -t^{a_-}$, respectively,
obtained by least-squares fitting to the computed data. 
The fitting is based on the values of $\emax(t)$ and $\einf(t)$ for $t\in T_{\pm}$. 
The horizontal axis indicates the number of data points included in the fit, 
beginning with the largest computed values of $t\in T_{\pm}$. 
Figures~\ref{fig:emax-exp} and  \ref{fig:einf-exp} indicates that $a_+<0.259$ and $a_-<0.261$.

Similarly, Figures~\ref{fig:emax-log} and  \ref{fig:einf-log} display the exponents $b_+$ and $b_-$
in the estimate $\emax(t)\sim t^{1/4}(\log t)^{b_+}$ and $\einf(t)\sim -t^{1/4}(\log t)^{b_-}$, respectively,
Figures~\ref{fig:emax-log} and  \ref{fig:einf-log} indicates that $b_+<0.33$ and $b_-<0.4$.

Using these values of $b_{\pm}$,
Figures~\ref{fig:E(t)-log-0.4} and \ref{fig:E(t)-log-3} show the graphs of
\begin{equation*}
 \frac{\emax(t)}{1+t^{1/4}(\log t)^{2/5}},
 \frac{\einf(t)}{1+t^{1/4}(\log t)^{2/5}}
 \quad{and}\quad
 \frac{\emax(t)}{1+t^{1/4}(\log t)^{1/3}},
 \frac{\einf(t)}{1+t^{1/4}(\log t)^{1/3}},
\end{equation*}
respectively, for $t\in T_{\pm}$.
The horizontal axis represents $t$ on a logarithmic scale.
These observations lead us to the following conjecture.
\begin{equation*}
 E(t)=O(t^{1/4}(\log t)^{2/5})
 \quad\text{or}\quad
 E(t)=O(t^{1/4}(\log t)^{1/3})
\quad\text{as}\quad t\to\infty.
\end{equation*}
For comparison, Figure~\ref{fig:E(t)-log-0.3} shows the graphs of
\begin{equation*}
 \frac{\emax(t)}{1+t^{1/4}(\log t)^{3/10}}
 \quad\text{and}\quad
 \frac{\einf(t)}{1+t^{1/4}(\log t)^{3/10}}.
\end{equation*}
This figure also suggests that
\begin{equation*}
 \emax(t)=O\bigl(t^{1/4}(\log t)^{3/10}\bigr)\quad\text{as}\quad t\to\infty.
\end{equation*}

\begin{figure}[H]
\centering
\includegraphics[width=0.6\linewidth]{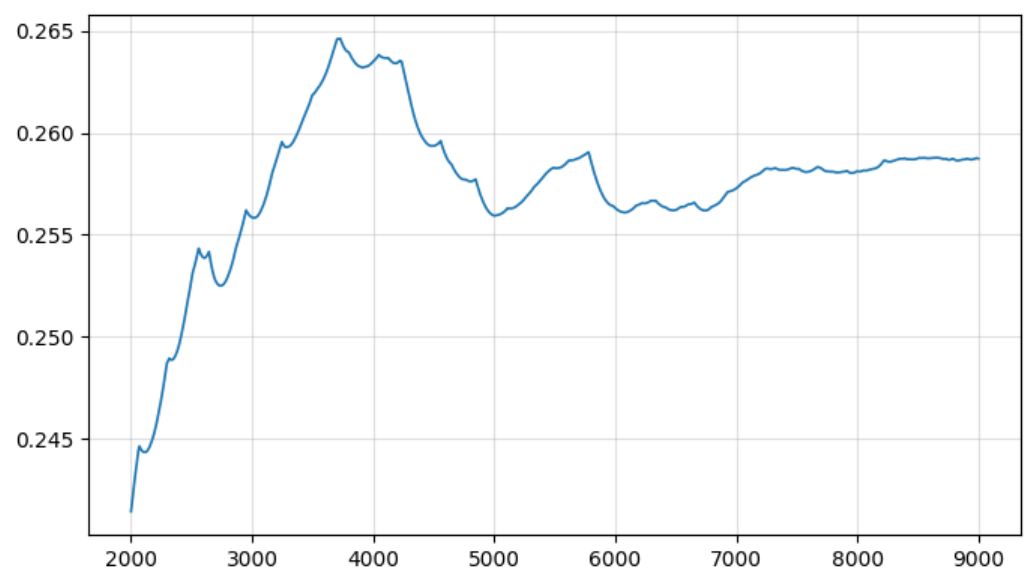}
\caption{Asymptotic exponent $a_+$ in $\emax(t)\sim t^{a_+}$}
\label{fig:emax-exp}
\end{figure}

\begin{figure}[H]
\centering
\includegraphics[width=0.6\linewidth]{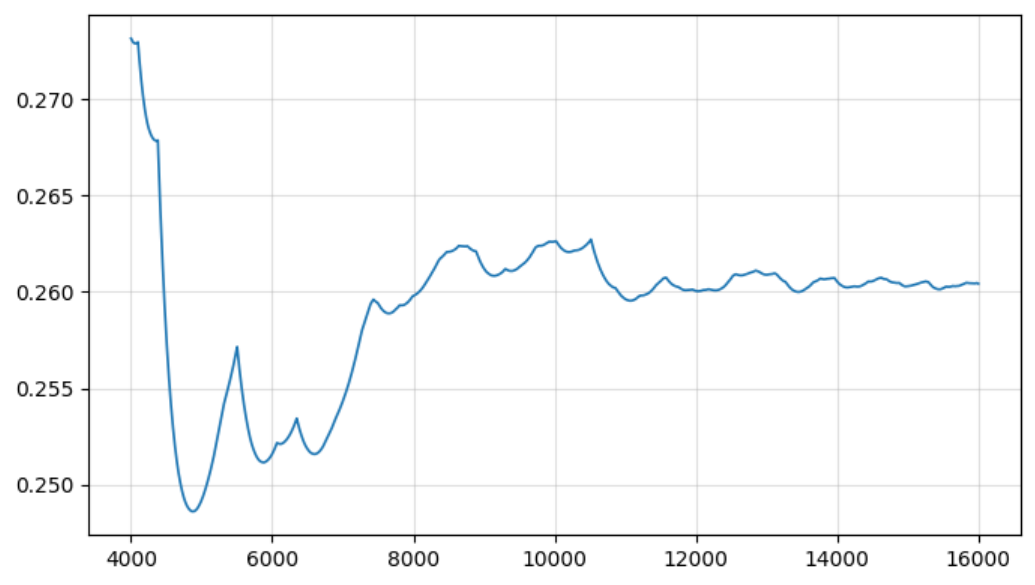}
\caption{Asymptotic exponent $a_-$ in $\einf(t)\sim -t^{a_-}$}
\label{fig:einf-exp}
\end{figure}

\begin{figure}[H]
\centering
\includegraphics[width=0.6\linewidth]{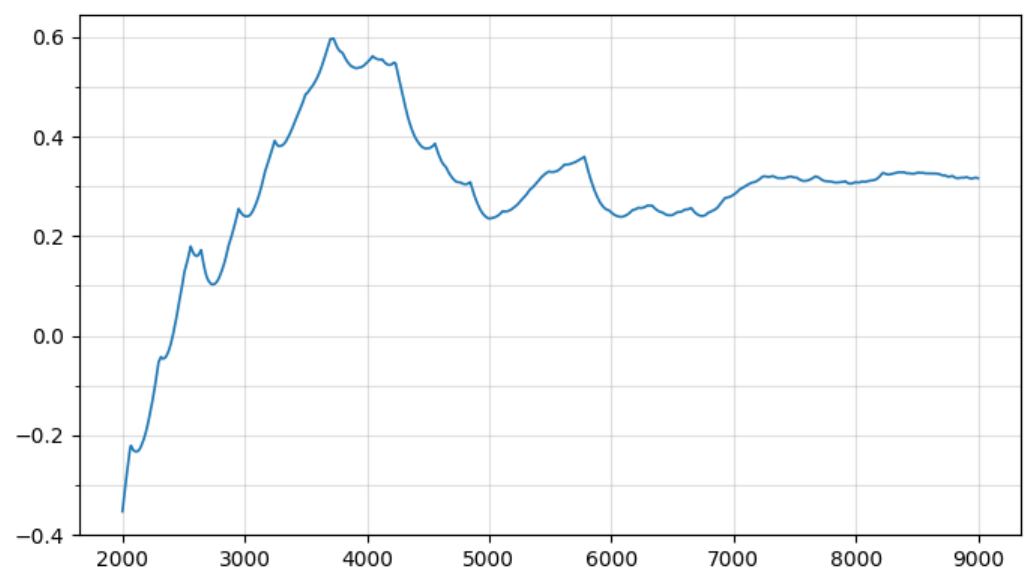}
\caption{Asymptotic exponent $b_+$ in $\emax(t)\sim t^{1/4}(\log t)^{b_+}$}
\label{fig:emax-log}
\end{figure}

\begin{figure}[H]
\centering
\includegraphics[width=0.6\linewidth]{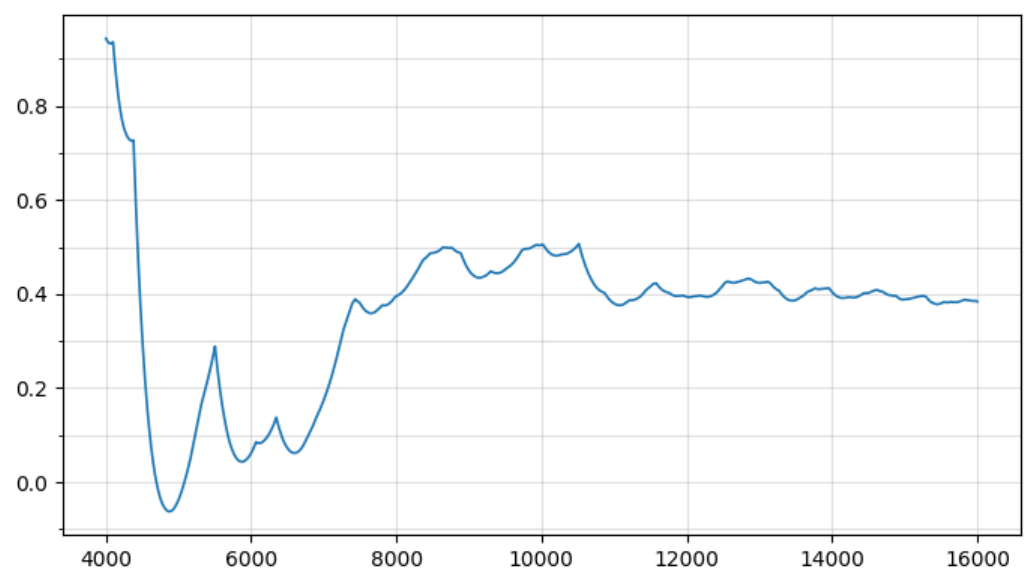}
\caption{Asymptotic exponent $b_-$ in $\einf(t)\sim -t^{1/4}(\log t)^{b_-}$}
\label{fig:einf-log}
\end{figure}

\begin{figure}[H]
\centering
\includegraphics[width=0.53\linewidth]{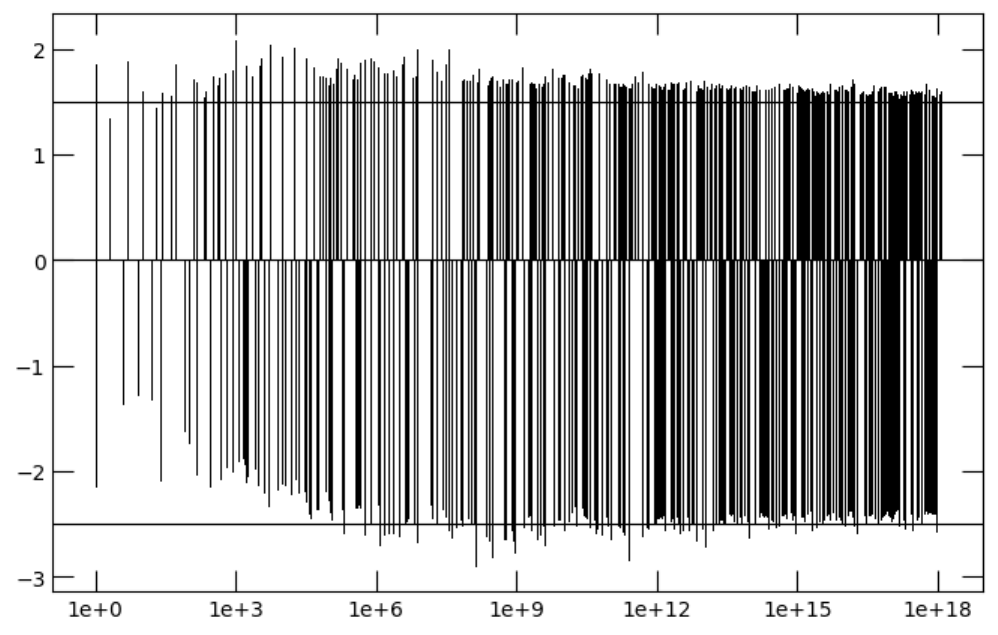}
\caption{$\emax(t)/(1+t^{1/4}(\log t)^{2/5})$, $\einf(t)/(1+t^{1/4}(\log t)^{2/5})$ 
            for $t\in T_\pm$}
\label{fig:E(t)-log-0.4}
\end{figure}

\begin{figure}[H]
\centering
\includegraphics[width=0.53\linewidth]{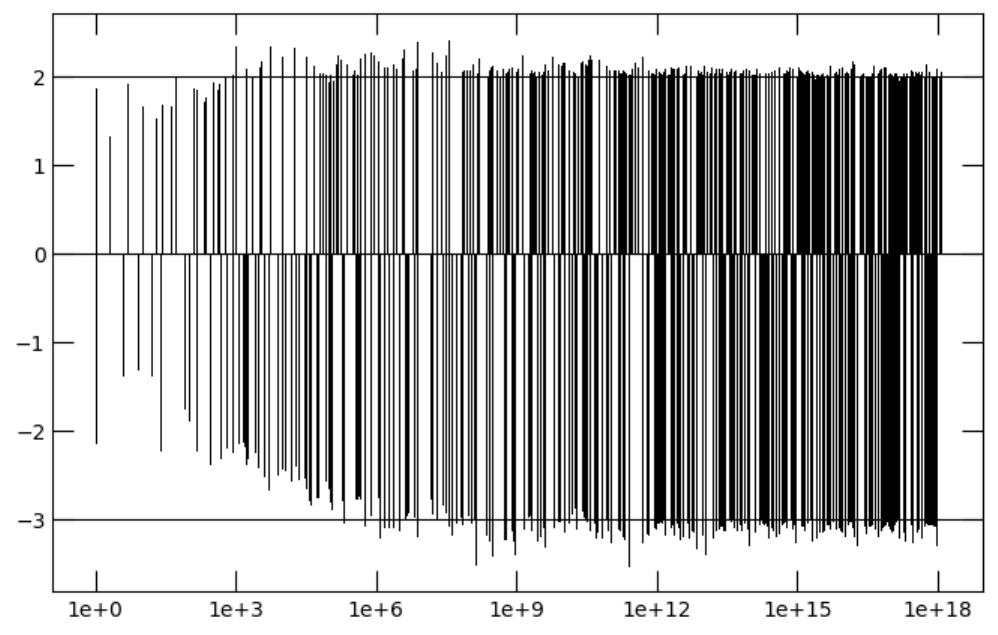}
\caption{$\emax(t)/(1+t^{1/4}(\log t)^{1/3})$, $\einf(t)/(1+t^{1/4}(\log t)^{1/3})$ 
            for $t\in T_\pm$}
\label{fig:E(t)-log-3}
\end{figure}

\begin{figure}[H]
\centering
\includegraphics[width=0.53\linewidth]{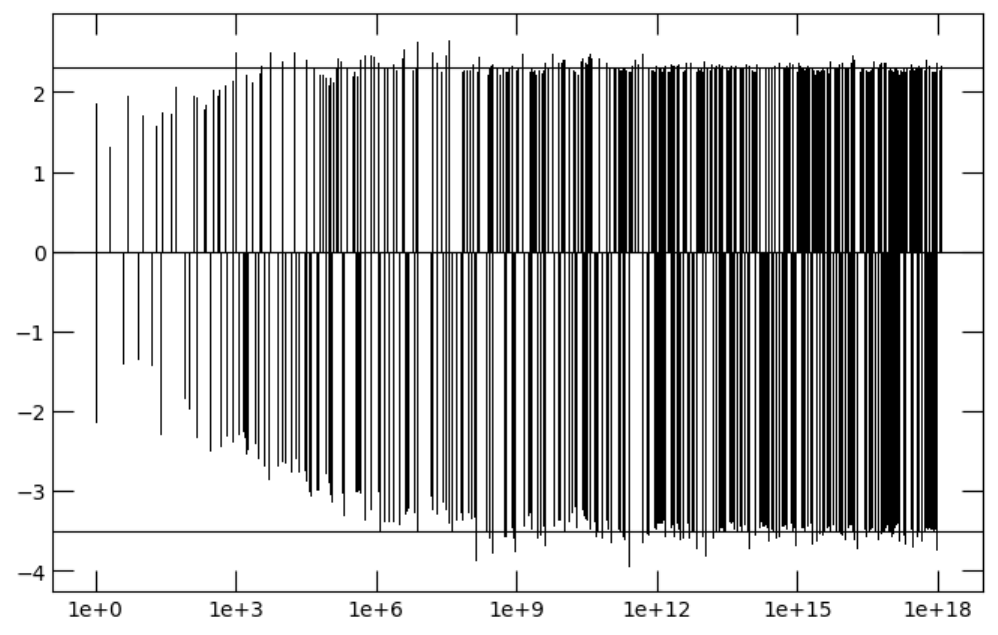}
\caption{$\emax(t)/(1+t^{1/4}(\log t)^{3/10})$, $\einf(t)/(1+t^{1/4}(\log t)^{3/10})$ 
            for $t\in T_\pm$}
\label{fig:E(t)-log-0.3}
\end{figure}

\section{The tendency of the points in $T_{\pm}$ to cluster}\label{sec:tendency}

We found that the sets $T_{\pm}$ are not uniformly distributed over the interval $[0,1.22\times 10^{18}]$ covered by our computations, 
but are instead concentrated in certain relatively narrow subintervals. 
More precisely, the following proposition holds.

\begin{prop}\label{prop:sqrt}
Let $S_{\max}$ and $S_{\inf}$ denote the sets of integers listed at the end of this section.
Then
\begin{align}
 T_+ \cap [0,1.22\times 10^{18}] & \subset\displaystyle\bigcup_{s\in S_{\max}}\Bigl((s-1)^2,s^2\Bigl],
\\ 
 T_- \cap [0,1.22\times 10^{18}] & \subset\displaystyle\bigcup_{s\in S_{\inf}}\Bigl((s-1)^2,s^2\Bigl].
\end{align}
\end{prop}

\begin{proof}
We have detected all $T_\pm$ over $[0,\ 1.22\times10^{18}]$ by the computation.
Elements of $S_{\max}$ and $S_{\inf}$ are decided by those $T_+$ and $T_-$
in the results, respectively.
\end{proof}

\begin{rem}
However, we were unable to identify any underlying pattern in the sequences $S_{\max}$ and $S_{\inf}$.
\end{rem}

Actually as pointed out in \cite{Tromp1990}, those points in the interval $\Bigl((s-1)^2,s^2\Bigl]$
for each $s\in S_{\max}$ or $s\in S_{\inf}$ are concentrated in a narrower interval.
Figure~\ref{fig:E(t)-sqrt} shows the graphs of
\begin{equation*}
 \frac{\emax(t)}{1+t^{1/4}(\log t)^{1/3}}
 \quad\text{and}\quad
 \frac{\einf(t)}{1+t^{1/4}(\log t)^{1/3}}
\end{equation*}
as functions of $\sqrt{t}-\lfloor\sqrt{t}\rfloor$, $\sqrt{t}\in S_{\max}\cup S_{\inf}$, where $\lfloor x \rfloor$ is the greatest integer less than or equal to $x$.
As can be seen from Figure~\ref{fig:E(t)-sqrt}, the majority of the elements of $T_+$ lie in intervals of the form
\begin{equation}
 \bigcup_{s\in S_{\max}}((s-1+0.2)^2,(s-1+0.5)^2],
\end{equation}
while those of $T_-$ lie in intervals of the form
\begin{equation}
 \bigcup_{s\in S_{\inf}}((s-1+0.9)^2,s^2].
\end{equation}
\begin{figure}[htbp]
\centering
\includegraphics[width=0.6\linewidth]{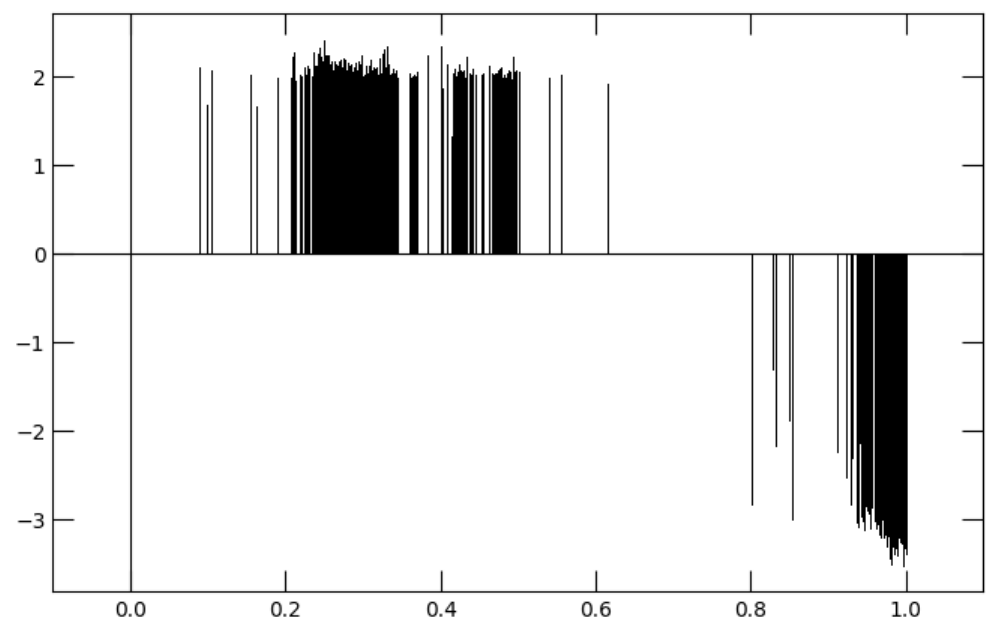}
\caption{$\emax(t)/(1+t^{1/4}(\log t)^{1/3})$, $\einf(t)/(1+t^{1/4}(\log t)^{1/3})$
            as functions of $\sqrt t -\lfloor \sqrt t \rfloor$}
\label{fig:E(t)-sqrt}
\end{figure}

More precisely, they are concentrated in much narrower intervals as follows. 
\begin{prop}\label{prop:sqrt 2}
Let $S_{\max}$ and $S_{\inf}$ be as in Proposition~\ref{prop:sqrt}.
Then there exist four functions $\kappa_+^i(s)$ on $S_{\max}$ and $\kappa_-^i(s)$ on $S_{\inf}$, $i=1,2$, with $0<\kappa_{\pm}^1(s)<\kappa_{\pm}^2(s)\le1$ such that 
\begin{align}
 T_+ \cap [0,1.22\times 10^{18}] & \subset \bigcup_{s\in S_{\max}}\Bigl( (s-1+\kappa_+^1(s))^2, (s-1+\kappa_+^2(s))^2 \Bigr],
\label{eqn:T+intervals}
\\
 T_- \cap [0,1.22\times 10^{18}] & \subset\bigcup_{s\in S_{\inf}}\Bigl( (s-1+\kappa_-^1(s))^2, (s-1+\kappa_-^2(s))^2 \Bigr].
\label{eqn:T-intervals}
\end{align}
\end{prop}

\begin{proof}
From our computations results $T_\pm$, we are able to decide those function values.
\end{proof}

\begin{rem}
(i) Narrowness is measured as
\begin{eqnarray*}
	&\kappa_+^2(s)-\kappa_+^1(s) \lessapprox2.4\times10^{-3}	
		\qquad&\mbox{if }s\in S_{\max},\ s\ge10^7	,	\\
	&\kappa_-^2(s)-\kappa_-^1(s) \lessapprox4.6\times10^{-3}		
		\qquad&\mbox{if }s\in S_{\inf},\ s\ge10^7.
\end{eqnarray*}
(ii) For example, for the largest element $s=1087048060\in S_{\max}$,
223 elements of $T_+$ are concentrated in $\kappa_+^1(s)=0.248$, 
$\kappa_+^2(s)=0.249$.
For the largest element $s=954266569\in S_{\inf}$,
1423 elements of $T_-$ are concentrated in $\kappa_-^1(s)=0.996$, 
$\kappa_-^2(s)=0.998$.
\end{rem}

Based on these observation, 
we made two all-in-one programs that detect all elements of $T_{\pm} \cap [0,1.22\times 10^{18}]$ (see Listing 5 and Listing 6 in Appendix).
These programs call Subroutine \subroutineexplore\ about intervals given
in the right hand side of (\ref{eqn:T+intervals}) and (\ref{eqn:T-intervals}).

\begin{defn}
We define $S_{\max}$ to be the set consisting of the 394 integers:
\begin{align*}
	&\{1,2,3,4,5,6,7,8,12,13,15,16,19,21,25,30,32,42,48,57,59,75,99,132,\\
	&\quad182,215,216,249,254,264,290,317,322,349,370,373,385,423,484,563,\\
	&\qquad\cdots\cdots,\\
	&\quad\qquad8874677,8917058,9704937,10320942,10552690,10778479,\\
	&\qquad\qquad\cdots\cdots,\\
	&\quad\qquad\qquad951504555,957864634,1046272275,1082245583,1087048060\}
\end{align*}
We also define $S_{\inf}$ to be the set consisting of the 310 integers:
\begin{align*}
	&\{1,2,3,4,5,9,10,12,17,22,25,29,34,38,39,41,43,51,55,63,72,89,99,\\
	&\quad106,122,123,140,147,149,169,178,195,201,229,241,284,292,309,326,\\
	&\qquad\cdots\cdots,\\
	&\quad\qquad9467773,9640452,11065910,12198321,12487831,\\
	&\qquad\qquad\cdots\cdots,\\
	&\quad\qquad\qquad865552004,890451476,917810515,921058948,954266569\}
\end{align*}
(For all elements of $S_{\max}$ and $S_{\inf}$, see Listing \ref{emaxallinone} and Listing \ref{einfallinone} in Appendix, respectively.)
\end{defn}

\section*{Acknowledgement}

This work was supported by JSPS KAKENHI Grant Numbers JP21K03304, JP26K06858 and JP24H00186.
The computation was performed using the following supercomputers: 
the supercomputer Fugaku provided by the RIKEN Center for Computational Science (Project ID: hp210224),
the FUJITSU Server PRIMERGY CX2550 M7 (Miyabi-C) at Joint Center for Advanced High Performance Computing (JCAHPC),
and the FUJITSU Supercomputer PRIMEHPC FX1000 (Wisteria-O/BDEC-01) at the Information Technology Center, The University of Tokyo.


%
\section*{Appendix}
\begin{lstlisting}[caption=Program to output elements of $T_+$ less than 1.22e+18 ({\tt emax-allinone.f90}), label=emaxallinone]
! to compile: $gfortran emax-allinone.f90 explore.f90 gauss.f90
!
MODULE MY_PARAMETERS
  IMPLICIT NONE
  INTEGER(4), PARAMETER:: INCSIZE = 10000000
  INTEGER(4), PARAMETER:: NSHIFT = 0
END MODULE MY_PARAMETERS
!
PROGRAM EMAX_ALLINONE
IMPLICIT NONE
INTEGER(8),PARAMETER::SMAX(125+49+61+73+86) = (/ INTEGER(8) :: &
 1,2,3,4,5,6,7,8,12,13,15,16,19,21,25,30,32,42,48,57,59,75,99,132,&
 182,215,216,249,254,264,290,317,322,349,370,373,385,423,484,563,&
 576,619,671,677,759,882,952,1034,1232,1237,1329,1541,1642,1877,&
 1889,1927,1956,2439,2656,2758,2766,3960,4422,4909,5544,5858,8231,&
 8614,9225,9356,9399,9938,9981,10068,10631,12009,12390,12593,15282,&
 15350,15591,16020,16036,16412,16702,17381,19766,21132,22168,22288,&
 24196,25511,25813,26054,26509,27704,28559,31124,33054,35825,35895,&
 36468,36798,43949,44813,46365,49058,51421,54403,54486,58292,58628,&
 60778,61172,62008,62046,62452,66896,68728,75665,76109,88093,93757,&
 96535,98933,& ! 125 s
 100949,115767,125472,125713,126552,131484,140158,143246,151546,&
 153967,155650,156324,172781,177711,186518,189094,200074,236517,&
 238016,242810,291942,312173,341986,347802,365081,381398,403896,&
 407052,412008,431076,448803,465391,466694,491426,496585,510518,&
 516833,533520,573550,576240,586776,632957,693903,697170,817323,&
 879728,912699,937433,961664,& ! 49 s
 1048240,1072261,1090856,1138857,1170063,1193902,1325800,1337538,&
 1382652,1390826,1443541,1469549,1513281,1601653,1645364,1671305,&
 1865140,1902384,1949273,2027042,2031448,2034782,2040704,2257920,&
 2309085,2617544,2703056,2840786,2982750,2983473,3082616,3159839,&
 3185570,3426661,3633991,3832688,4162665,4316815,4321221,4528725,&
 4827482,4932553,4988331,5510906,5519914,5709855,5915063,6005562,&
 6240661,6550677,6867357,7576407,7597250,7669336,7805040,7911139,&
 8017669,8363090,8874677,8917058,9704937,& ! 61 s
 10320942,10552690,10778479,10830817,11102220,11533985,11667185,&
 12389965,12966892,13287219,13661256,13845576,14273761,14418452,&
 15305574,15352076,15478477,16455756,16647015,16713635,16716274,&
 16920058,17991920,18060844,19916208,20015205,22569199,22628871,&
 23465932,24006081,25030239,25248346,26207117,27786737,31037611,&
 31409233,32120248,32340685,33092857,35229289,36040233,37159190,&
 38871541,38913198,39264478,40815277,42434273,45133991,47230671,&
 49397437,50100801,50220225,51512807,53624643,53955687,54714351,&
 55597284,56995667,58436926,59215998,59224428,60618434,64328988,&
 65233335,66374148,68846165,70016900,78241285,82545198,85285467,&
 87793746,89206305,92885279,& ! 73 s
 100156181,105522101,110724886,116525007,116867205,122850703,&
 128658021,147619479,148034618,148954544,150659324,152661417,&
 160763037,167626738,168160562,172268824,177294864,179678067,&
 181759366,187625596,191723979,194143130,203114702,205995390,&
 230212685,235280508,236243033,244188193,269969949,276197017,&
 301276495,312225241,322387373,336771615,341057503,344825803,&
 347471544,355204966,357151795,373129150,391163605,392312915,&
 396745187,401597124,417481473,421786470,423778503,425167643,&
 426333117,451050418,451852895,456970323,467011952,468302207,&
 496764787,506958444,507232938,535432138,564607540,578849975,&
 585238919,597800970,618421789,622507778,638807341,656498944,&
 669422659,711526674,716949896,718085898,730251044,736328906,&
 742425935,753411046,813706018,869454130,889413934,898447215,&
 904014319,924091638,944263039,951504555,957864634,1046272275,&
 1082245583,1087048060 /) ! 86 S
REAL(8),PARAMETER::KAPPA1_TABLE(73+86) = (/ REAL(8) :: &
 0.34,0.232,0.335,0.266,0.421,0.436,0.308,0.258,0.258,0.253,&
 0.261,0.238,0.256,0.254,0.264,0.451,0.415,0.248,0.247,0.259,&
 0.241,0.313,0.249,0.265,0.252,0.263,0.25,0.332,0.274,0.24,&
 0.264,0.248,0.265,0.243,0.226,0.271,0.249,0.258,0.328,0.258,&
 0.253,0.277,0.226,0.272,0.326,0.266,0.253,0.331,0.266,0.244,&
 0.279,0.263,0.249,0.25,0.265,0.243,0.467,0.28,0.255,0.246,&
 0.258,0.251,0.285,0.474,0.323,0.295,0.227,0.258,0.277,0.424,&
 0.273,0.323,0.25,& ! for 73 s
 0.265,0.245,0.265,0.342,0.278,0.262,0.331,0.256,0.298,0.255,&
 0.268,0.263,0.296,0.264,0.288,0.275,0.296,0.317,0.249,0.244,&
 0.263,0.256,0.272,0.249,0.243,0.368,0.249,0.261,0.27,0.249,&
 0.288,0.251,0.275,0.25,0.265,0.286,0.423,0.251,0.252,0.317,&
 0.259,0.489,0.271,0.332,0.36,0.292,0.249,0.426,0.313,0.266,&
 0.30,0.25,0.25,0.305,0.423,0.431,0.341,0.276,0.288,0.286,&
 0.265,0.25,0.279,0.243,0.249,0.287,0.248,0.252,0.268,0.418,&
 0.22,0.467,0.288,0.282,0.252,0.254,0.429,0.25,0.274,0.249,&
 0.335,0.261,0.278,0.31,0.256,0.247 /) ! for 86 s
INTEGER(8):: FIRST_T, LAST_T, S
INTEGER:: I, NEXS
REAL(8):: EMAX=0.0, EINF=0.0, KAPPA1, KAPPA2
REAL:: TIME0, TIME1
!
CALL CPU_TIME(TIME0)
PRINT *, "PROGRAM EMAX_ALLINONE"
DO I = 1, SIZE(SMAX)
   S = SMAX(I)
   IF (I>125+49+61) THEN ! If S>10^7
      KAPPA1 = KAPPA1_TABLE(I-125-49-61) ; KAPPA2 = KAPPA1+0.004
   ELSE IF (S < 1000000) THEN ! If S<10^6
      KAPPA1 = 0.08 ; KAPPA2 = 1.0
   ELSE ! If 10^6<S<10^7
      KAPPA1 = 0.23 ; KAPPA2 = 0.51
   END IF
   FIRST_T = (REAL(S-1,8)+KAPPA1) ** 2.0_8 ! REDUCE BY KAPPA FUNCTIONS
   LAST_T = (REAL(S-1,8)+KAPPA2) ** 2.0_8
   PRINT *,S,":",FIRST_T,LAST_T
   
   CALL EXPLORE(FIRST_T, LAST_T, EMAX, EINF, .TRUE., .FALSE., NEXS)
   IF (NEXS <= 0) THEN
      PRINT *,"ERROR: INVALID INTERVAL"
      STOP
   END IF
END DO

CALL CPU_TIME(TIME1)
PRINT *,"TOTALLY, CPU TIME: ", TIME1-TIME0, " SECONDS."
END PROGRAM EMAX_ALLINONE
\end{lstlisting}
\begin{lstlisting}[caption=Program to output elements of $T_-$ less than
1.22e+18 ({\tt einf-allinone.f90}), label=einfallinone]
PROGRAM EINF_ALLINONE
IMPLICIT NONE
INTEGER(8),PARAMETER::SINF(99+36+52+56+67) = (/ INTEGER(8) :: &
 1,2,3,4,5,9,10,12,17,22,25,29,34,38,39,41,43,51,55,63,72,89,99,&
 106,122,123,140,147,149,169,178,195,201,229,241,284,292,309,326,&
 338,432,449,601,627,647,678,758,881,1059,1072,1202,1364,1516,1762,&
 2028,2127,2168,2520,2728,3783,3889,3957,4401,5164,5227,5473,5922,&
 6371,7129,7887,8399,9582,10217,11018,11069,11535,14775,15936,&
 16139,17660,23273,24637,27691,29055,30559,37447,41971,43352,45182,&
 52311,52654,55848,61321,63897,78981,79592,87153,88843,91571,& !99 s
 103309,103618,112116,123253,124926,133093,135872,158247,158556,&
 166308,167515,172679,177906,192152,211804,217861,228824,253787,&
 257879,292284,298020,301697,318442,323606,389333,390540,403507,&
 424187,447460,492642,696701,701853,768787,799346,919714,956398,&!36
 1016355,1060559,1068205,1107926,1151278,1210757,1391362,1428495,&
 1476562,1509289,1649241,1694423,1770816,1861180,2113137,2141434,&
 2163792,2396276,2463210,2656434,3018851,3293115,3335912,3343664,&
 4000784,4238741,4602365,4820226,4945135,4966993,5038087,5114031,&
 5156125,5389867,5907455,6035266,6153545,6206424,6359966,6500026,&
 6635017,6687328,7065575,7812762,7962168,8027586,8502877,9091542,&
 9184050,9378930,9467773,9640452,& ! 52 s
 11065910,12198321,12487831,13078017,13260710,13756114,13916017,&
 14035313,14321337,14354933,14745524,15141986,15629172,16905173,&
 17530010,18470472,19907016,20044218,21574523,22295040,23020475,&
 24095763,24341243,26741937,27162778,27486087,28732135,29585673,&
 30036776,34551662,35070170,35860709,36997212,40700408,41066311,&
 44315367,45419604,45575724,49818731,53466663,57462906,58949482,&
 61182987,65833122,68206769,69725124,74652925,80134380,81264821,&
 82678855,86525810,89483352,93985136,95122711,95891498,99956488,&!56
 109335418,114725285,116745453,121089939,125564807,128068484,&
 136122088,136485712,166564184,171305021,184550664,191119954,&
 193881684,198867460,209481644,210423110,211195866,214901189,&
 227729214,244531730,254566611,257620833,268192737,277999497,&
 284741017,290330030,293217427,296923957,305193592,315229502,&
 318637063,325719849,328834363,346613064,360568201,362035753,&
 366353323,373976953,377475081,378143012,383610276,387234643,&
 425164671,444386596,511786335,512205049,524607231,528154645,&
 595386869,628457551,658947417,722899886,725557078,730522309,&
 754651613,761669295,776980317,816586053,817487052,829366931,&
 836797960,844247704,865552004,890451476,917810515,921058948,&
 954266569 /) ! 67 s
REAL(8),PARAMETER::KAPPA1_TABLE(56+67) = (/ REAL(8) ::&
 0.995,0.996,0.991,0.987,0.997,0.991,0.996,0.997,0.97,0.997,&
 0.978,0.978,0.981,0.997,0.993,0.993,0.986,0.995,0.962,0.995,&
 0.997,0.98,0.994,0.996,0.968,0.984,0.987,0.99,0.99,0.995,&
 0.944,0.975,0.991,0.995,0.996,0.97,0.96,0.992,0.965,0.997,&
 0.996,0.993,0.995,0.997,0.992,0.997,0.988,0.988,0.992,0.951,&
 0.995,0.988,0.996,0.997,0.994,0.997,& !! for 56 s
 0.997,0.997,0.997,0.995,0.971,0.997,0.994,0.996,0.977,0.995,&
 0.997,0.997,0.987,0.963,0.986,0.994,0.987,0.959,0.996,0.995,&
 0.995,0.992,0.995,0.996,0.988,0.996,0.997,0.997,0.99,0.997,&
 0.987,0.993,0.994,0.994,0.987,0.996,0.99,0.997,0.991,0.935,&
 0.991,0.997,0.996,0.993,0.99,0.994,0.995,0.996,0.997,0.992,&
 0.988,0.979,0.994,0.994,0.997,0.974,0.996,0.996,0.991,0.994,&
 0.996,0.994,0.995,0.988,0.995,0.979,0.994 /) ! for 67 s
INTEGER(8):: FIRST_T, LAST_T, S
INTEGER:: I,NEXS
REAL(8):: EMAX=0.0, EINF=0.0, KAPPA1, KAPPA2
REAL:: TIME0, TIME1
!
CALL CPU_TIME(TIME0)
PRINT *, "PROGRAM EINF_ALLINONE"
DO I = 1, SIZE(SINF)
   S = SINF(I)
   IF (I>99+36+52) THEN ! IF S>10-7
      KAPPA1 = KAPPA1_TABLE(I-99-36-52)
      KAPPA2 = MIN(KAPPA1+0.005, 1.0)
   ELSEIF (S <= 999999) THEN
      KAPPA1 = 0.8_8 ; KAPPA2 = 1.0_8
   ELSE IF (S <= 9999999) THEN
      KAPPA1 = 0.85_8 ; KAPPA2 = 1.0_8
   END IF
   FIRST_T = (REAL(S-1,8)+KAPPA1) ** 2.0_8 ! REDUCE BY KAPPA FUNCTIONS
   LAST_T = (REAL(S-1,8)+KAPPA2) ** 2.0_8
   PRINT *,S,":",FIRST_T,LAST_T

   CALL EXPLORE(FIRST_T, LAST_T, EMAX, EINF, .FALSE., .TRUE., NEXS) ! EINF ONLY
   IF (NEXS <= 0) THEN
      PRINT *,"ERROR: INVALID INTERVAL"
      STOP
   END IF
END DO
  
CALL CPU_TIME(TIME1)
PRINT *,"TOTALLY, CPU TIME: ", TIME1-TIME0, " SECONDS."
END PROGRAM EINF_ALLINONE
\end{lstlisting}

\section*{Electric Appendix}
\begin{itemize}
\item Job data table
\item Programs (the same as Listing 1 to Listing 6)
\end{itemize}
(Before publication, these are available in
{\tt http://fjord.} {\tt sci.ibaraki.ac.jp/gausscp/, user=pcf, pwd=8700}.)

\vskip 20mm

\small

\noindent
{Satoshi Yamaguchi}\\
{Graduate School of Economics, Hitotsubashi University, 2-1 Naka, Kunitachi, Tokyo 186-8601, Japan}\\
{satoshiyamaguchi.1998@gmail.com}
\\[1ex]
{Shoichi Fujima}\\
{Department of Mathematics, Ibaraki University, Mito, Ibaraki 310-8512, Japan}\\
{shoichi.fujima.sci@vc.ibaraki.ac.jp}
\\[1ex]
{Shigehiko Kuratsubo}\\
{Department of Mathematical Sciences, Hirosaki University, Hirosaki 036-8561, Japan}\\
{kuratubo@hirosaki-u.ac.jp}
\\[1ex]
{Eiichi Nakai}\\
{Department of Mathematics, Ibaraki University, Mito, Ibaraki 310-8512, Japan}\\
{eiichi.nakai.math@vc.ibaraki.ac.jp}
\\[1ex]
{Tsuyoshi Yoneda}\\
{Graduate School of Economics, Hitotsubashi University, 2-1 Naka, Kunitachi, Tokyo 186-8601, Japan}\\
{t.yoneda@r.hit-u.ac.jp}

\end{document}